\documentclass[reqno]{amsart}
\usepackage{amssymb,graphicx}%
\usepackage{ifpdf}
\ifpdf
 \usepackage[hyperindex]{hyperref}
\else
 \expandafter\ifx\csname dvipdfm\endcsname\relax
 \usepackage[hypertex,hyperindex]{hyperref}
 \else
 \usepackage[dvipdfm,hyperindex]{hyperref}
 \fi
\fi
\allowdisplaybreaks[4]
\theoremstyle{plain}
\newtheorem{thm}{Theorem}[section]
\newtheorem{lem}{Lemma}[section]
\newtheorem{cor}{Corollary}[section]
\theoremstyle{remark}
\newtheorem{rem}{Remark}[section]

\numberwithin{equation}{section}

\begin{document}

\title[New gradient estimates]
{\textbf{S\MakeLowercase{ome new  gradient estimates for two nonlinear parabolic equations under} R\MakeLowercase{icci flow}}}

\author[W. Wang]{Wen Wang\qquad Hui Zhou}
\address[W. Wang\quad H. Zhou]{School of Mathematics and Statistics, Hefei Normal
University, Hefei 230601,P.R.China;School of mathematical Science, University of Science and
Technology of China, Hefei 230026, China}
\email{\href{mailto: W. Wang <wwen2014@mail.ustc.edu.cn>}{wwen2014@mail.ustc.edu.cn}}

\today
\begin{abstract}
In this paper, by  maximum principle and cutoff function, we investigate gradient estimates for positive
 solutions to two nonlinear parabolic equations under Ricci flow. The related Harnack inequalities
are deduced. An result about positive solutions on closed manifolds under Ricci flow is abtained.
As applications, gradient estimates and Harnack inequalities for positive solutions to the heat
equation under Ricci flow are derived. These results in the paper can be regard as generalizing the gradient
estimates of Li-Yau, J. Y. Li, Hamilton and Li-Xu to the Ricci flow. Our results also improve the estimates
of S. P. Liu and J. Sun to the nonlinear parabolic equation under Ricci flow.
\end{abstract}



\thanks{Corresponding author: Wen Wang, E-mail: wwen2014@mail.ustc.edu.cn}

\thanks{This paper was typeset using \AmS-\LaTeX}

\maketitle

\section{\textbf{Introduction}}
Beginning with the pioneering work of Li and Yau \cite{14}, gradient estimates are also known as differential
Harnack inequalities, which have tremendous impact in geometric analysis,
as shown for example in \cite{14,15,16}. Moreover, both have very important applications  in singularity analysis.
In perelman's geometrization conjecture \cite{22,23}   on the poincar\'{e} conjecture, a differential Harnack inequality played an
important role.

Next, we simply introduce research progress  associated with this article.

Let $(M^{n}, g)$ be a complete Riemannian manifold. Li and Yau \cite{14}  established a famous
gradient estimate for positive solutions to the following heat equation
 \begin{equation}\label{1.1}
u_{t}=\Delta u
\end{equation}
on $(M^{n}, g)$, which is described as

\textbf{Theorem A}\quad (Li-Yau \cite{14}) \emph{Let $(M^{n}, g)$ be a complete
Riemannian manifold. Suppose that on the ball $B_{2R}$,
$Ricci(B_{2R})\geq -K$. Then for any $\alpha >1$,
\begin{equation}\label{1.2}
\sup_{B_{2R}}\left(\frac{|\nabla u|^2}{u^2}-\alpha
\frac{u_t}{u}\right)\leq
\frac{C\alpha^2}{R^2}\left(\frac{\alpha^2}{\alpha^{2}-1}+\sqrt{K}R\right)
+\frac{n\alpha^{2}K}{\alpha-1}+\frac{n\alpha^{2}}{2t}.
\end{equation}
In general, on a complete Riemannian manifold, if $Ricci(M)\geq -k$,
by letting $R\rightarrow\infty$ in (1.2), one inferred
\begin{equation}\label{1.3}
\frac{|\nabla u|^2}{u^2}-\alpha
\frac{u_t}{u}\leq\frac{n\alpha^{2}k}{2(\alpha-1)}+
\frac{n\alpha^{2}}{2t}.
\end{equation}}

In $1991$, Li \cite{15} generalized Li and Yau's estimates to the nonlinear parabolic equation
\begin{equation}\label{1.4}
\left(\Delta -\frac{\partial}{\partial t}\right)u(x,t) +
h(x,t)u^{\alpha}(x,t)=0
\end{equation}
on $(M^{n}, g)$.
In $1993$, Hamilton in \cite{8}  generalized the constant $\alpha$ of Li and Yau's result to the function $\alpha(t)=e^{2Kt}$.
In $2006$, Sun \cite{27} also obtained a  gradient estimate of different coefficient.
In $2011$, Li and Xu in \cite{17} further promoted Li and Yau's result, and found  two new functions $\alpha(t)$.
Recently, first author and Zhang in \cite{28} further generalized Li and Xu's results to the nonlinear parabolic equation ~\eqref{1.4}.
Related results can be found in \cite{5,11,32}.

In this paper, we investigate the two nonlinear parabolic equations
\begin{equation}\label{1.5}
\partial_{t}u(x,t)=\Delta u(x,t)+h(x,t) u^{l}(x,t)
\end{equation}
and
\begin{equation}\label{1.6}
\partial_{t}u(x,t)=\Delta u(x,t)+au(x,t)\log u(x,t)
\end{equation}
under Ricci flow, where the function $h(x,t)\geq 0$ is defined on
$M^{n}\times[0,T]$, which is $C^2$ in the first variable and $C^1$
in the second variable, $T$ is a positive constant and $l, a\in \mathbb{R}$, respectively.

Recently, there are a number of studies on Ricci flow on manifolds by R. Hamilton \cite{9,10} and others, because
the Ricci flow is a powerful tool in analyzing the structure of manifolds. Assume  $M^n$ is an $n$-dimensional manifold without boundary,
and let $(M^{n}, g(t))_{t\in[0,T]}$  be an $n$-dimensional complete manifold with
metric $g(t)$ evolving by the Ricci flow
\begin{equation}\label{1.7}
\frac{\partial g(t)}{\partial t}=-2R_{ij},\quad (x,t)\in M^{n}\times[0,T].
\end{equation}

In 2008, Kuang and Zhang \cite{11} proved a gradient estimate for positive solutions to the conjugate heat equation under
Ricci flow on a closed manifold. In 2009,  Liu \cite{18} derived a gradient estimate for positive solutions to the heat equation
under Ricci flow. Afterwards, Sun\cite{26} generalized Liu's results to general geometric flow.
 In 2010, Bailesteanu, Cao and Pulemotov \cite{1} established some gradient estimates for positive solutions
to the heat equation under Ricci flow. In 2016, Li and Zhu \cite{19} generalized J. Y. Li's  \cite{15} estimates
under Ricci flow.
Recently, Cao and Zhu \cite{3} derived some Aronson and B\'{e}nilan estimates for porous medium equation
\begin{equation*}
u_{t}=\Delta u^{m},\quad m>1
\end{equation*}
 under Ricci flow.
 Li, Bai and Zhang \cite{13} studied
fast diffusion equation
\begin{equation*}
u_{t}=\Delta u^{m},\quad 0<m<1
\end{equation*}under the
Ricci flow. Zhao and Fang \cite{31} generalized Yang's result \cite{30} to the Ricci flow.

Firstly, we introduce three $C^1$ functions $\alpha(t)$, $\varphi(t)$ and $\gamma(t)$ $: (0,+\infty)\rightarrow (0,+\infty)$.
Suppose that three $C^1$ functions $\alpha(t)$, $\varphi(t)$ and $\gamma(t)$
  satisfy the following conditions:\\
  $(C1)$~ $\alpha(t)>1$, $\varphi(t)$ and $\gamma(t)$.\\
  $(C2)$~ $\alpha(t)$ and $\varphi(t)$ satisfy the following system
 \begin{equation*}
\left\{\aligned
\frac{2\varphi}{n}-2\alpha K\geq (\frac{2\varphi}{n}-\alpha')\frac{1}{\alpha},\\
\frac{2\varphi}{n}-\alpha'>0,\\
\frac{\varphi^2}{n}+\alpha\varphi'\geq 0.
\endaligned\right.
\end{equation*}
$(C3)$ ~$\gamma(t)$ satisfies
 \begin{equation*}
\frac{\gamma'}{\gamma}-(\frac{2\varphi}{n}-\alpha')\frac{1}{\alpha}\leq 0.
\end{equation*}
 $(C4)$~ $\gamma(t)$ is non-decreasing, and $\alpha(t)$ is also non-decreasing or is bounded uniformly.

 This paper is organized as follows: We prove  gradient estimates for the equation ~\eqref{1.5} in
Section $2$ and  gradient estimates for the equation ~\eqref{1.6} in Section $3$. We derive related Harnack inequalities
in Section $4$. As special case, we deduce gradient estimates and Harnack inequality to the heat equation in section $5$.
Detailed calculation of some specific functions $\alpha(t)$, $\varphi(t)$  and $\gamma(t)$ are given in section $6$.

\section{\textbf{Gradient estimates for the equation ~\eqref{1.5}}}
In this section, we will derive some new gradient estimates  for positive solutions
to  equation ~\eqref{1.5} under the Ricci flow.

\subsection{\textbf{Main results}}
\

We state our results as follows.

\begin{thm}
Let $(M^{n}, g(t))_{t\in [0,T]}$ be a complete solution to the Ricci flow ~\eqref{1.7}.
  Assume that $|\mathrm{Ric}(x,t)|\leq K$ for some $K>0$ and all $t\in [0,T]$.
  Suppose that there exist three functions $\alpha(t)$, $\varphi(t)$ and $\gamma(t)$
  satisfy  conditions (C1), (C2), (C3) and (C4).

Given $x_{0}\in M^n$ and $R>0$, let $u$ be a positive solution of the  equation
~\eqref{1.5}
in the cube $B_{2R,T}:=\{(x,t)|d(x,x_{0},t)\leq 2R, 0\leq t\leq T\}$. Let $h(x,t)$ be a function defined
on $M^{n}\times[0,T]$ which is $C^1$ in $t$ and $C^2$ in $x$, satisfying $|\nabla h|^{2}\leq \delta_{2}h$ and
$\Delta h\geq -\delta_{3}$ on $B_{2R,T}$ for some positive constants $\delta_{2}$ and $\delta_{3}$.

$(1)$ $l\leq 1$.\quad  If $\frac{\gamma\alpha^4}{\alpha-1}\leq C_{1}$ for some constant $C_1$,
then
\begin{align*}
\nonumber
&\frac{|\nabla u|^2}{u^2}-\alpha\frac{u_{t}}{u}+\alpha h(x,t)u^{l-1}\\
&\leq
C\alpha^{2}\left(\frac{1}{R^2}+\frac{\sqrt{K}}{R}+K\right)+\frac{Cn^{2}\alpha^4}{R^{2}\gamma}+n^{\frac{3}{2}}\alpha^{2}K\\
&+\alpha\sqrt{n\overline{u}_{1}\delta_3}
+n\alpha^{2}\overline{u}_{1}\delta_{1}+\sqrt{\frac{2-l}{2}}
\alpha^{\frac{3}{2}} \sqrt{n\overline{u}_{1}\delta_{2}}+\alpha\varphi.
\end{align*}

If $\frac{\gamma}{\alpha-1}\leq C_{2}$ for some constant $C_2$,
then
\begin{align*}
\nonumber
&\frac{|\nabla u|^2}{u^2}-\alpha\frac{u_{t}}{u}+\alpha h(x,t)u^{l-1}\\
&\leq
C\alpha^{2}\left(\frac{1}{R^2}+\frac{\sqrt{K}}{R}+K\right)+\frac{Cn^{2}\alpha^{4}}{R^{2}\gamma}+n^{\frac{3}{2}}\alpha^{2}K\\
&+\alpha\sqrt{n\overline{u}_{1}\delta_3}
+n\alpha^{2}\overline{u}_{1}\delta_{1}+\sqrt{\frac{2-l}{2}}
\alpha^{\frac{3}{2}} \sqrt{n\overline{u}_{1}\delta_{2}}+\alpha\varphi,
\end{align*}
where $C$ is a positive constant depending only on $n$ and set
$$\overline{u}_{1}:=\max_{B_{2R,T}}u^{l-1},\quad \delta_{1}:=\max_{B_{2R,T}}h(x,t).$$

$(2)$ $l> 1$.\quad  If $\frac{\gamma\alpha^4}{\alpha-1}\leq C_{1}$ for some constant $C_1$,
then
\begin{align*}
\nonumber
&\frac{|\nabla u|^2}{u^2}-\alpha\frac{u_{t}}{u}+\alpha h(x,t)u^{l-1}\\
&\leq
C\alpha^{2}\left(\frac{1}{R^2}+\frac{\sqrt{K}}{R}+K\right)+\frac{Cn^{2}\alpha^4}{R^{2}\gamma}+n^{\frac{3}{2}}\alpha^{2}K
+n\alpha^{2}(l-1)\delta_{1}\overline{u}_{2}\\
&+\alpha\sqrt{\frac{n(l\alpha-1)\overline{u}_{2}\delta_{2}}{l-1}}
+\alpha^{\frac{3}{2}}\sqrt{n(l-1)\delta_{1}\varphi}+\alpha^{\frac{3}{2}}\sqrt{n\delta_{3}\overline{u}_{2}}+\alpha\varphi.
\end{align*}

If $\frac{\gamma}{\alpha-1}\leq C_{2}$ for some constant $C_2$,
then
\begin{align*}
\nonumber
&\frac{|\nabla u|^2}{u^2}-\alpha\frac{u_{t}}{u}+\alpha h(x,t)u^{l-1}\\
&\leq
C\alpha^{2}\left(\frac{1}{R^2}+\frac{\sqrt{K}}{R}+K\right)+\frac{Cn^{2}\alpha^{4}}{R^{2}\gamma}
+n^{\frac{3}{2}}\alpha^{2}K+n\alpha^{2}(l-1)\delta_{1}\overline{u}_{2}\\
&
+\alpha\sqrt{\frac{n(l\alpha-1)\overline{u}_{2}\delta_{2}}{l-1}}
+\alpha^{\frac{3}{2}}\sqrt{n(l-1)\delta_{1}\varphi}+\alpha^{\frac{3}{2}}\sqrt{n\delta_{3}\overline{u}_{2}}+\alpha\varphi,
\end{align*}
where $C$ is a positive constant depending only on $n$ and set
$$\overline{u}_{2}:=\max_{B_{2R,T}}u^{l-1},\quad \delta_{1}:=\max_{B_{2R,T}}h(x,t).$$
\end{thm}

Let us list some examples to illustrate the Theorem $2.1$ holds for different circumstances and
see appendix in section $6$ for detailed calculation process.

\begin{cor}  Suppose that $(M^{n}, g(t))_{t\in[0, T]}$ satisfies the hypotheses of Theorem $2.1$.
 Then the following special estimates are valid.

1. Li-Yau type:
$$\alpha(t)=constant,\quad \varphi(t)=\frac{\alpha n}{t}+\frac{nK\alpha^2}{\alpha-1}, \gamma(t)=t^{\theta}
\quad  with \quad 0<\theta\leq 2.$$
If $l\leq 1$, then
\begin{align*}
\nonumber
&\frac{|\nabla u|^2}{u^2}-\alpha\frac{u_{t}}{u}+\alpha h(x,t) u^{l-1}\\
&\leq
C\alpha^{2}\left[\frac{1}{R^2}(1+\sqrt{K}R)+\frac{\alpha^{2}}{\alpha-1}\frac{1}{R^2}+K\right]+\alpha\varphi\\
&
+n^{\frac{3}{2}}\alpha^{2}K+\alpha\sqrt{n\overline{u}_{1}\delta_3}
+n\alpha^{2}\overline{u}_{1}\delta_{1}+\sqrt{\frac{2-l}{2}}
\alpha^{\frac{3}{2}} \sqrt{n\overline{u}_{1}\delta_{2}}.
\end{align*}
If $l>1$, then
\begin{align*}
&\frac{|\nabla u|^2}{u^2}-\alpha\frac{u_{t}}{u}+\alpha h(x,t) u^{l-1}\\
&\leq
C\alpha^{2}\left[\frac{1}{R^2}(1+\sqrt{K}R)+\frac{\alpha^{2}}{\alpha-1}\frac{1}{R^2}+K\right]+\alpha\varphi\\
&
+n^{\frac{3}{2}}\alpha^{2}K
+n\alpha^{2}(l-1)\delta_{1}\overline{u}_{2}+\alpha\sqrt{\frac{n(l\alpha-1)\overline{u}_{2}\delta_{2}}{l-1}}\\
&
+\alpha^{\frac{3}{2}}\sqrt{n(l-1)\delta_{1}\varphi}+\alpha^{\frac{3}{2}}\sqrt{n\delta_{3}\overline{u}_{2}}.
\end{align*}
2. Hamilton type:
$$\alpha(t)=e^{2Kt},\quad \varphi(t)=\frac{n}{t}e^{4Kt},\quad \gamma(t)=te^{2Kt}.$$
If $l\leq 1$, then
\begin{align*}
&\frac{|\nabla u|^2}{u^2}-\alpha\frac{u_{t}}{u}+\alpha h(x,t)u^{l-1}\\
&\leq
C\alpha^{2}\left[\frac{1}{R^2}(1+\sqrt{K}R)+K\right]+\frac{C\alpha^4}{R^{2}te^{2Kt}}+\alpha\varphi\\
&
+n^{\frac{3}{2}}\alpha^{2}K+\alpha\sqrt{n\overline{u}_{1}\delta_3}
+n\alpha^{2}\overline{u}_{1}\delta_{1}+\sqrt{\frac{2-l}{2}}
\alpha \sqrt{n\overline{u}_{1}\delta_{2}}.
\end{align*}
If $l>1$, then
\begin{align*}
&\frac{|\nabla u|^2}{u^2}-\alpha\frac{u_{t}}{u}+\alpha h(x,t)u^{l-1}\\
&\leq
C\alpha^{2}\left[\frac{1}{R^2}(1+\sqrt{K}R)+K\right]+\frac{C\alpha^4}{R^{2}te^{2Kt}}+\alpha\varphi\\
&
+n^{\frac{3}{2}}\alpha^{2}K
+n\alpha^{2}(l-1)\delta_{1}\overline{u}_{2}+\alpha\sqrt{\frac{n(l\alpha-1)\overline{u}_{2}\delta_{2}}{l-1}}\\
&
+\alpha^{\frac{3}{2}}\sqrt{n(l-1)\delta_{1}\varphi}+\alpha^{\frac{3}{2}}\sqrt{n\delta_{3}\overline{u}_{2}}..
\end{align*}
3. Li-Xu type:
\begin{align*}&\alpha(t)=1+\frac{\sinh(Kt)\cosh(Kt)-Kt}{\sinh^{2}(Kt)},\quad \varphi(t)=2nK[1+\coth(Kt)],\\
& \gamma(t)=\tanh(Kt).
\end{align*}
If $l\leq 1$, then
\begin{align*}
&\frac{|\nabla u|^2}{u^2}-\alpha\frac{u_{t}}{u}+\alpha h(x,t) u^{l-1}\\
&\leq
C\left[\frac{1}{R^2}(1+\sqrt{K}R)+K\right]+\frac{C}{R^{2}\tanh(Kt)}+\alpha\varphi\\
&
+n^{\frac{3}{2}}\alpha^{2}K+\alpha\sqrt{n\overline{u}_{1}\delta_3}
+n\alpha^{2}\overline{u}_{1}\delta_{1}+\sqrt{\frac{2-l}{2}}
\alpha\sqrt{n\overline{u}_{1}\delta_{2}}.
\end{align*}
If $l>1$, then
\begin{align*}
&\frac{|\nabla u|^2}{u^2}-\alpha\frac{u_{t}}{u}+\alpha h(x,t)u^{l-1}\\
\leq&
C\alpha^{2}\left[\frac{1}{R^2}(1+\sqrt{K}R)+K\right]+\frac{C}{R^{2}\tanh(Kt)}+\alpha\varphi\\
&
+n^{\frac{3}{2}}\alpha^{2}K
+n\alpha^{2}(l-1)\delta_{1}\overline{u}_{2}+\alpha\sqrt{\frac{n(l\alpha-1)\overline{u}_{2}\delta_{2}}{l-1}}\\
&
+\alpha^{\frac{3}{2}}\sqrt{n(l-1)\delta_{1}\varphi}+\alpha^{\frac{3}{2}}\sqrt{n\delta_{3}\overline{u}_{2}},
\end{align*}
where $\alpha(t)$ is bounded uniformly.

4. Linear Li-Xu type:
$$\alpha(t)=1+2Kt, \varphi(t)=\frac{n}{t}+nK(1+2Kt+\mu Kt),\gamma(t)=Kt \quad with \quad \mu\geq \frac{1}{4}.$$
If $l\leq 1$, then
\begin{align*}
&\frac{|\nabla u|^2}{u^2}-\alpha\frac{u_{t}}{u}+\alpha h(x,t) u^{l-1}\\
&\leq
C\alpha^{2}\left[\frac{1}{R^2}(1+\sqrt{K}R)+K\right]+\frac{C\alpha^4}{R^{2}Kt}+\alpha\varphi\\
&
+n^{\frac{3}{2}}\alpha^{2}K+\alpha\sqrt{n\overline{u}_{1}\delta_3}
+n\alpha^{2}\overline{u}_{1}\delta_{1}+\sqrt{\frac{2-l}{2}}
\alpha\sqrt{n\overline{u}_{1}\delta_{2}}.
\end{align*}
If $l>1$, then
\begin{align*}
&\frac{|\nabla u|^2}{u^2}-\alpha\frac{u_{t}}{u}+\alpha h(x,t)u^{l-1}\\
&\leq
C\alpha^{2}\left[\frac{1}{R^2}(1+\sqrt{K}R)+K\right]+\frac{C\alpha^4}{R^{2}Kt}+\alpha\varphi\\
&
+n^{\frac{3}{2}}\alpha^{2}K
+n\alpha^{2}(l-1)\delta_{1}\overline{u}_{2}+\alpha\sqrt{\frac{n(l\alpha-1)\overline{u}_{2}\delta_{2}}{l-1}}\\
&
+\alpha^{\frac{3}{2}}\sqrt{n(l-1)\delta_{1}\varphi}+\alpha^{\frac{3}{2}}\sqrt{n\delta_{3}\overline{u}_{2}}.
\end{align*}
\end{cor}

\begin{rem}
The above results can be regard as generalizing the gradient
estimates of Li-Yau \cite{14}, J. Y. Li \cite{15}, Hamilton \cite{8} and Li-Xu \cite{17} to the Ricci flow. Our results also generalize the estimates
of S. P. Liu \cite{18} and J. Sun \cite{26} to the nonlinear parabolic equation under the Ricci flow.
\end{rem}

The local estimates in Theorem $2.1$ imply global estimates.
\begin{cor}
Let $(M^{n}, g(t))_{t\in [0,T]}$ be a complete solution to the Ricci flow ~\eqref{1.7}.
  Assume that $|\mathrm{Ric}(x,t)|\leq K$ for some $K>0$ and all $(x,t)\in M^{n}\times[0,T]$.
Let $u(x,t)$ be a positive solution to equation ~\eqref{1.5} on $M^{n}\times [0,T]$.
Let $h(x,t)$ be a function defined
on $M^{n}\times[0,T]$ which is $C^1$ in $t$ and $C^2$ in $x$, satisfying $|\nabla h|^{2}\leq \delta_{2}h$ and
$\Delta h\geq -\delta_{3}$ on $M^{n}\times[0,T]$ for some positive constants $\delta_{2}$ and $\delta_{3}$.

If $l\leq 1$ and for $(x,t)\in M^{n}\times (0,T]$, then
\begin{align*}
\nonumber
&\frac{|\nabla u|^2}{u^2}-\alpha\frac{u_{t}}{u}+\alpha h(x,t)u^{l-1}\\
&\leq \alpha\varphi+
C\alpha\left[\alpha K+\sqrt{\overline{u}_{1}\delta_{3}}+\alpha\overline{u}_{1}\delta_{1}
+\sqrt{\frac{2-l}{2}}\sqrt{\overline{u}_{1}\delta_{2}}\right],
\end{align*}
where where $C$ is a positive constant depending only on $n$ and set
$$\overline{u}_{1}:=\max_{M^{n}\times [0,T]}u^{l-1},\quad \delta_{1}:=\max_{M^{n}\times [0,T]}h(x,t).$$

If $l>1$ and for $(x,t)\in M^{n}\times (0,T]$, then
\begin{align*}
&\frac{|\nabla u|^2}{u^2}-\alpha\frac{u_{t}}{u}+\alpha h(x,t)u^{l-1}\leq \alpha\varphi\\
\leq&
C\alpha\Big[\alpha K+(l-1)\alpha\overline{u}_{2}\delta_{1}+\sqrt{\frac{(l\alpha-1)\overline{u}_{2}\delta_{2}}{l-1}}
+\alpha^{\frac{1}{2}}\sqrt{(l-1)\delta_{1}\varphi}+\alpha^{\frac{1}{2}}\sqrt{\overline{u}_{2}\delta_{3}}\Big],
\end{align*}
where where $C$ is a positive constant depending only on $n$ and set
$$\overline{u}_{1}:=\max_{M^{n}\times [0,T]}u^{l-1},\quad \delta_{1}:=\max_{M^{n}\times [0,T]}h(x,t).$$
\end{cor}

We can derive a gradient estimate for an any positive solution to the following nonlinear parabolic
equation under the Ricci flow on a closed manifold without any curvature conditions.
The method of the proof is inspired by Hamilton \cite{10}, Shi \cite{23} and Liu \cite{18}.

\begin{thm}
Let $(M^{n}, g(x,t))_{t\in [0,T]}$ be a  solution to the Ricci flow ~\eqref{1.7} on a closed manifold.
 If $u$ is a positive solution to equation
 $$\partial_{t}u=\Delta u+h(t)u^{l-1},$$
 where $h(t)$ is a $C^{1}$ function and $h(t)\leq 0$.
 Then for $l\geq 1$, we have
 \begin{equation}\label{2.1}
 |\nabla u(x,t)|^{2}\leq\frac{1}{2t}\left(\max_{x\in M^{n}}u^{2}(x,0)-u^{2}(x,t)\right)\quad for\quad
 (x,t)\in M^{n}\times[0,T].
 \end{equation}
 \end{thm}

\subsection{\textbf{Auxilliary lemma}}
\

To prove main results, we need a lemma.

Let $f=\ln u$. Then
\begin{equation}\label{2.2}
f_{t}=\Delta f+|\nabla f|^{2}+hu^{l-1}.
\end{equation}
Let $F=|\nabla f|^{2}-\alpha f_{t}+\alpha hu^{l-1}-\alpha\varphi$, where $\alpha=\alpha(t)>1$ and $\varphi=\varphi(t)>0$.

\begin{lem} Suppose that $(M^{n}, g(t))_{t\in[0, T]}$ satisfies the hypotheses of Theorem $2.1$.
We also assume that $\alpha(t)>1$ and $\varphi(t)>0$ satisfy the following system
\begin{equation}\label{2.3}
\left\{\aligned
\frac{2\varphi}{n}-2\alpha K\geq (\frac{2\varphi}{n}-\alpha')\frac{1}{\alpha},\\
\frac{2\varphi}{n}-\alpha'>0,\\
\frac{\varphi^2}{n}+\alpha\varphi'\geq 0,
\endaligned\right.
\end{equation}
and $\alpha(t)$ is non-decreasing.
Then
\begin{eqnarray}\label{2.4}
\nonumber
(\Delta-\partial_{t}) F&\geq &|f_{ij}+\frac{\varphi}{n}g_{ij}|^{2}+(\frac{2\varphi}{n}-\alpha')\frac{1}{\alpha}F
-\alpha^{2}n^{2}K^{2}-2\nabla f\nabla F\\
\nonumber
&&+2c(\alpha-1)(l-1)u^{l-1}|\nabla f|^{2}+\alpha(l-1)^{2}hu^{l-1}|\nabla f|^{2}\\
&&+\alpha(l-1)hu^{l-1}\Delta f+\alpha u^{l-1}\Delta h+2(\alpha-1)u^{l-1}\nabla f\cdot\nabla h.
\end{eqnarray}
\end{lem}

\begin{proof}[\textbf{Proof}]
By directly computing,  we have
\begin{eqnarray*}
\Delta F&=&\Delta|\nabla f|^{2}-\alpha\Delta(f_{t})+\alpha \Delta(hu^{l-1})\\
&=&2|f_{ij}|^{2}+2f_{j}f_{iij}+2R_{ij}f_{i}f_{j}-\alpha\Delta(f_t)+\alpha h\Delta(u^{l-1})\\
&&+\alpha u^{l-1}\Delta h+2\alpha\nabla h\nabla u^{l-1}\\
&=&2\Big(|f_{ij}|^{2}+\alpha R_{ij}f_{ij}\Big)+2f_{j}f_{iij}+2R_{ij}f_{i}f_{j}
-\alpha(\Delta f)_{t}\\
&&+\alpha h \Delta(u^{l-1})+\alpha u^{l-1}\Delta h+2\alpha\nabla h\nabla u^{l-1},
\end{eqnarray*}
where we have used Bochner's formula and
\begin{equation*}
\Delta(f_t)=(\Delta f)_{t}-2\sum_{i,j=1}^{n}R_{ij}f_{ij}.
\end{equation*}
Applying Young's inequality
$$R_{ij}f_{ij}\leq |R_{ij}||f_{ij}|\leq \frac{\alpha}{2}|R_{ij}|^{2}+\frac{1}{2\alpha}|f_{ij}|^{2},$$
we conclude for $|R_{ij}|\leq K$,
\begin{eqnarray}\label{2.5}
\nonumber
\Delta F
&\geq&|f_{ij}|^{2}-\sum\alpha^{2} |R_{ij}|^{2}+2f_{j}f_{iij}+2R_{ij}f_{i}f_{j}
-\alpha(\Delta f)_{t}\\
\nonumber
&&+\alpha h \Delta(u^{l-1})+\alpha u^{l-1}\Delta h++2\alpha\nabla h\nabla u^{l-1}\\
\nonumber
&\geq&|f_{ij}|^{2}-\alpha^{2}n^{2}K^{2}+2f_{j}f_{iij}+2R_{ij}f_{i}f_{j}
-\alpha(\Delta f)_{t}\\
&&+\alpha h \Delta(u^{l-1})+\alpha u^{l-1}\Delta h+2\alpha\nabla h\nabla u^{l-1}.
\end{eqnarray}
On the other hand, we infer
\begin{eqnarray}\label{2.6}
\nonumber
\partial_{t} F&=&(|\nabla f|^{2})_{t}-\alpha f_{tt}-\alpha' f_{t}+\alpha' h u^{l-1}+\alpha h(u^{l-1})_{t}\\
\nonumber
&&+\alpha u^{l-1}h_{t}-\alpha \varphi'-\alpha'\varphi\\
\nonumber
&=&2\nabla f\nabla(f_t)+2R_{ij} f_{i}f_{j}-\alpha f_{tt}-\alpha'f_{t}+\alpha' h u^{l-1}+\alpha u^{l-1}h_{t}\\
&&+\alpha h (u^{l-1})_{t}-\alpha \varphi'-\alpha'\varphi.
\end{eqnarray}
We follow from ~\eqref{2.5} and ~\eqref{2.6},
\begin{eqnarray*}
(\Delta-\partial_{t}) F&\geq &|f_{ij}|^{2}-\alpha^{2}n^{2}K^{2}+2\nabla f\nabla(\Delta f)
-\alpha(\Delta f)_{t}+\alpha h \Delta(u^{l-1})\\
&&+\alpha u^{l-1}\Delta h+2\alpha\nabla h\nabla u^{l-1}-2\nabla f\nabla(f_t)
+\alpha f_{tt}+\alpha'f_{t}\\
&&-\alpha' h u^{l-1}-\alpha h (u^{l-1})_{t}
-\alpha u^{l-1}h_{t}+\alpha \varphi'+\alpha'\varphi\\
&=&|f_{ij}|^{2}-\alpha^{2}n^{2}K^{2}+2\nabla f\nabla(\Delta f)
-\alpha(f_{t}-|\nabla f|^{2}-hu^{l-1})_{t}\\
&&+\alpha h \Delta(u^{l-1})+\alpha u^{l-1}\Delta h+2\alpha\nabla h\nabla u^{l-1}
-2\nabla f\nabla(f_t)+\alpha f_{tt}\\
&&+\alpha'f_{t}-\alpha' h u^{l-1}
-\alpha h (u^{l-1})_{t}-\alpha u^{l-1}h_{t}+\alpha \varphi'+\alpha'\varphi\\
&=&|f_{ij}|^{2}-\alpha^{2}n^{2}K^{2}+2\nabla f\nabla(\Delta f)
+\alpha(|\nabla f|^{2})_{t}+\alpha h \Delta(u^{l-1})\\
&&+\alpha u^{l-1}\Delta h+2\alpha\nabla h\nabla u^{l-1}-2\nabla f\nabla(f_t)+\alpha'f_{t}\\
&&-\alpha' h u^{l-1}+\alpha \varphi'+\alpha'\varphi.
\end{eqnarray*}
By using the formula
\begin{equation*}
(|\nabla f|^2)_{t}=2\nabla f\cdot\nabla (f_{t})+2\mathrm{Ric}(\nabla f, \nabla f),
\end{equation*}
we obtain
\begin{eqnarray}\label{2.7}
\nonumber
(\Delta-\partial_{t}) F&\geq &|f_{ij}|^{2}-\alpha^{2}n^{2}K^{2}+2\nabla f\nabla(\Delta f)
+2\alpha\nabla f\nabla(f_t)\\
\nonumber
&&+2\alpha R_{ij}f_{i}f_{j}
+\alpha h \Delta(u^{l-1})+\alpha u^{l-1}\Delta h+2\alpha\nabla h\nabla u^{l-1}\\
\nonumber
&&-2\nabla f\nabla(f_t)+\alpha'f_{t}-\alpha' h u^{l-1}+\alpha \varphi'+\alpha'\varphi\\
\nonumber
&=&|f_{ij}|^{2}+2\alpha R_{ij}f_{i}f_{j}-\alpha^{2}n^{2}K^{2}-2\nabla f\nabla F\\
\nonumber
&&+2(\alpha-1)\nabla f\nabla(h u^{l-1})+\alpha h \Delta(u^{l-1})+\alpha u^{l-1}\Delta h\\
&&+2\alpha\nabla h\nabla u^{l-1}+\alpha'f_{t}-\alpha' h u^{l-1}+\alpha \varphi'+\alpha'\varphi.
\end{eqnarray}
Applying the following two equations
\begin{eqnarray*}
\nabla(u^{l-1})&=&(l-1)u^{l-1}\nabla f,\\
\Delta(u^{l-1})&=&(l-1)^{2}u^{l-1}|\nabla f|^{2}+(l-1)u^{l-1}\Delta f,
\end{eqnarray*}
to ~\eqref{2.7}, we have
\begin{eqnarray}\label{2.8}
\nonumber
(\Delta-\partial_{t}) F&\geq &|f_{ij}|^{2}+2\alpha R_{ij}f_{i}f_{j}-\alpha^{2}n^{2}K^{2}+2\nabla f\nabla F
+\alpha u^{l-1}\Delta h
\\
\nonumber
&&+2h(\alpha-1)(l-1)u^{l-1}|\nabla f|^{2}+2\big[(\alpha-1)+\alpha(l-1)\big]u^{l-1}\nabla f\cdot\nabla h\\
\nonumber
&&+h\alpha(l-1)^{2}u^{l-1}|\nabla f|^{2}+h\alpha(l-1)u^{l-1}\Delta f\\
&&+\alpha'f_{t}-\alpha' c u^{l-1}+\alpha \varphi'+\alpha'\varphi.
\end{eqnarray}
Further applying  unit matrix $(\delta_{ij})_{n\times n}$ and ~\eqref{2.8}, we derive
\begin{eqnarray*}
(\Delta-\partial_{t}) F&\geq &|f_{ij}+\frac{\varphi}{n}\delta_{ij}|^{2}-2\alpha K|\nabla f|^{2}-\alpha^{2}n^{2}K^{2}+2\nabla f\nabla F
\\
&&+2h(\alpha-1)(l-1)u^{l-1}|\nabla f|^{2}+2\big[(\alpha-1)+\alpha(l-1)\big]u^{l-1}\nabla f\cdot\nabla h\\
&&+h\alpha(l-1)^{2}u^{l-1}|\nabla f|^{2}+h\alpha(l-1)u^{l-1}\Delta f+\alpha u^{l-1}\Delta h\\
&&+\alpha'f_{t}-\alpha' c u^{l-1}+\alpha \varphi'+\alpha'\varphi-\frac{\varphi^2}{n}-2\frac{\varphi}{n}\Delta f.
\end{eqnarray*}
Applying ~\eqref{2.2}, we have
\begin{eqnarray}\label{2.9}
\nonumber
(\Delta-\partial_{t}) F&\geq &|f_{ij}+\frac{\varphi}{n}\delta_{ij}|^{2}+(\frac{2\varphi}{n}-2\alpha K)|\nabla f|^{2}-(\frac{2\varphi}{n}-\alpha')f_{t}\\
\nonumber
&&+(\frac{2\varphi}{n}-\alpha')cu^{l-1}-(\frac{2\varphi}{n}-\alpha')\frac{\alpha\varphi}{\alpha}-\alpha^{2}n^{2}K^{2}-2\nabla f\nabla F\\
\nonumber
&&+2h(\alpha-1)(l-1)u^{l-1}|\nabla f|^{2}+2\big[(\alpha-1)+\alpha(l-1)\big]u^{l-1}\nabla f\cdot\nabla h\\
\nonumber
&&+h\alpha(l-1)^{2}u^{l-1}|\nabla f|^{2}+h\alpha(l-1)u^{l-1}\Delta f+\alpha u^{l-1}\Delta h\\
&&+\alpha \varphi'+\alpha'\varphi
-\frac{\varphi^2}{n}+(\frac{2\varphi}{n}-\alpha')\frac{\alpha\varphi}{\alpha}.
\end{eqnarray}
Therefore, ~\eqref{2.4} is derived from ~\eqref{2.3} and ~\eqref{2.9}. The proof is complete.
\end{proof}

\subsection{\textbf{Proof of Theorem $2.1$ and $2.2$}}
\

In this section, we will prove the Theorem $2.1$ and $2.2$.

\begin{proof}[\textbf{Proof of Theorem $2.1$}]
Let $G=\gamma(t)F$ and $\gamma(t)>0$ be non-decreasing. Then
\begin{align}\label{2.10}
\nonumber
(\Delta-\partial_{t}) G=&\gamma(\Delta-\partial_{t})F-\gamma'F\\
\nonumber
\geq &\gamma|f_{ij}+\frac{\varphi}{n}\delta_{ij}|^{2}+(\frac{2\varphi}{n}-\alpha')\frac{1}{\alpha}G
-\gamma\alpha^{2}n^{2}K^{2}-2\nabla f\nabla G\\
\nonumber
&+2h\gamma(\alpha-1)(l-1)u^{l-1}|\nabla f|^{2}+2(l\alpha-1)\gamma u^{l-1}\nabla f\cdot\nabla h\\
\nonumber
&+h\gamma\alpha(l-1)^{2}u^{l-1}|\nabla f|^{2}+h\gamma\alpha(l-1)u^{l-1}\Delta f+\alpha\gamma u^{l-1}\Delta h-\gamma' F\\
\nonumber
=&\gamma|f_{ij}+\frac{\varphi}{n}g_{ij}|^{2}+\Big[(\frac{2\varphi}{n}-\alpha')\frac{1}{\alpha}-\frac{\gamma'}{\gamma}\big]G
-\gamma\alpha^{2}n^{2}K^{2}-2\nabla f\nabla G\\
\nonumber
&+2h\gamma(\alpha-1)(l-1)u^{l-1}|\nabla f|^{2}+2\big[(\alpha-1)+\alpha(l-1)\big]\gamma u^{l-1}\nabla f\cdot\nabla h\\
&+\gamma\alpha(l-1)^{2}hu^{l-1}|\nabla f|^{2}+\gamma\alpha(l-1)hu^{l-1}\Delta f+\alpha\gamma u^{l-1}\Delta h.
\end{align}

Now let $\varphi(r)$ be a $C^2$ function on $[0,\infty)$ such that
\begin{equation*}
\varphi(r)=\left\{\aligned
1 \quad if ~r\in[0,1],\\
0 \quad if ~r\in [2,\infty),
\endaligned\right.
\end{equation*}
and
$$0\leq \varphi(r)\leq 1, \quad \varphi'(r)\leq 0,\quad \varphi''(r)\leq 0, \quad
\frac{|\varphi'(r)|}{\varphi(r)}\leq C,$$
where $C$ is an absolute constant.
Let define by
$$\phi(x,t)=\varphi(d(x,x_{0},t))=\varphi\left(\frac{d(x,x_{0},t)}{R}\right)
=\varphi\left(\frac{\rho(x,t)}{R}\right),$$
where $\rho(x,t)=d(x,x_{0},t)$.
By using  maximum principle, the argument of Calabi ~\cite{2} allows us to
suppose that the function $\phi(x,t)$ with support in $B_{2R,T}$, is $C^2$ at the maximum point.
By utilizing the Laplacian theorem, we deduce that
\begin{eqnarray}\label{2.11}
\frac{|\nabla \phi|^{2}}{\phi}\leq \frac{C}{R^2},\quad
-\Delta \phi\leq \frac{C}{R^2}(1+\sqrt{K}R),
\end{eqnarray}

For any $0\leq T_{1}\leq T$, let $H=\phi G$ and $(x_{1},t_{1})$ be the point in $B_{2R,T_1}$
at which $H$ attains its maximum value. We can suppose that
the value is positive, because otherwise the proof is trivial.
Then at the point $(x_{1}, t_{1})$, we infer
\begin{equation}\label{2.12}
\left.\begin{array}{rcl} 0=\nabla (\phi G)=G\nabla \phi +\phi\nabla
G,\vspace{2mm}
\\
\Delta(\phi G)\leq 0,\vspace{2mm}
\\
 \partial_t (\phi G)\geq 0.\end{array}\right\}
\end{equation}
By the evolution formula of the geodesic length under the Ricci flow ~\cite{6}, we calculate
\begin{align*}
\phi_{t}G=&-G\phi'\left(\frac{\rho}{R}\right)\frac{1}{R}\frac{d\rho}{dt}
=G\phi'\left(\frac{\rho}{R}\right)\int_{\gamma_{t_1}}\mathrm{Ric}(S,S)ds\\
\leq &G\phi'\left(\frac{\rho}{R}\right)\frac{1}{R}K\rho\leq G\phi'\left(\frac{\rho}{R}\right)K_{2}\leq G\sqrt{C}K,
\end{align*}
where $\gamma_{t_1}$ is the geodesic connecting $x$ and $x_0$ under the metric
$g(t_1)$, $S$ is the unite tangent vector to $\gamma_{t_1}$, and $ds$ is the
element of the arc length.

All the following computations are at the point $(x_{1}, t_{1})$.
It is not difficult to find that
\begin{eqnarray}\label{2.13}
\nonumber
|f_{ij}+\frac{\varphi}{n}\delta_{ij}|^{2}&\geq&\frac{1}{n}\Big(\mathbf{tr}|f_{ij}+\frac{\varphi}{n}\delta_{ij}| \Big)^{2}\\
\nonumber
&=&\frac{1}{n}\Big(\Delta f +\varphi\Big)\\
\nonumber
&=&\frac{1}{n}\Big[-\frac{1}{\alpha}F-\frac{1}{\alpha}(\alpha-1)|\nabla f|^{2}\Big]^{2}\\
&=&\frac{1}{\alpha^{2}n}\Big[\frac{G}{\gamma}+(\alpha-1)|\nabla f|^{2}\Big]^{2}.
\end{eqnarray}
and
\begin{eqnarray}\label{2.14}
\nonumber
\Delta f&=&f_{t}-|\nabla f|^{2}-cu^{l-1}\\
&=&-\frac{F}{\alpha}-\frac{\alpha-1}{\alpha}|\nabla f|^{2}-\varphi<0.
\end{eqnarray}

To obtain main results, two cases will be shown.

\textbf{Case $1$}\quad $l\leq 1$.\\
 From (2.14), we have $\Delta f\leq 0$. Then by substituting it into ~\eqref{2.10}, we obtain
\begin{align*}
\nonumber
(\Delta-\partial_{t}) G=&\gamma(\Delta-\partial_{t})F-\gamma'F\\
\nonumber
\geq &\gamma|f_{ij}+\frac{\varphi}{n}\delta_{ij}|^{2}+\Big[(\frac{2\varphi}{n}-\alpha')\frac{1}{\alpha}-\frac{\gamma'}{\gamma}\big]G
-\gamma\alpha^{2}n^{2}K^{2}-2\nabla f\nabla G\\
\nonumber
&+2h\gamma(\alpha-1)(l-1)u^{l-1}|\nabla f|^{2}+\alpha\gamma u^{l-1}\Delta h\\
&+2\big[(\alpha-1)+\alpha(l-1)\big]\gamma u^{l-1}\nabla f\cdot\nabla h,
\end{align*}
where we drop one term $h\gamma\alpha(l-1)^{2}u^{l-1}|\nabla f|^{2}$.
 Using ~\eqref{2.13}, we infer
\begin{align}\label{2.15}
\nonumber
0\geq&(\Delta-\partial_{t}) (\phi G)\\
\nonumber
=&G\Big(\Delta\phi-2\frac{|\nabla\phi|^2}{\phi}\Big)+\phi(\Delta-\partial_{t})G- G\phi_{t}\\
\nonumber
\geq&G\Big(\Delta\phi-2\frac{|\nabla\phi|^2}{\phi}\Big)+\frac{\phi\gamma}{\alpha^{2}n}\Big[\frac{G}{\gamma}+(\alpha-1)|\nabla f|^{2}\Big]^{2}\\
\nonumber
&+\Big[(\frac{2\varphi}{2}-\alpha')\frac{1}{\alpha}-\frac{\gamma'}{\gamma}\Big]\phi G-\gamma\phi\alpha^{2}n^{2}K^{2}-2\phi\nabla f\nabla G\\
\nonumber
&+2h\phi\gamma(\alpha-1)(l-1)u^{l-1}|\nabla f|^{2}+\phi\alpha\gamma u^{l-1}\Delta h\\
&+2\big[(\alpha-1)+\alpha(l-1)\big]\phi\gamma u^{l-1}\nabla f\cdot\nabla h-G\sqrt{C}K.
\end{align}
Multiply $\phi$ to inequality ~\eqref{2.15}, we have
\begin{eqnarray*}
\nonumber
0&\geq&\phi G\Big[\Delta\phi-2\frac{|\nabla\phi|^2}{\phi}+(\frac{2\varphi}{n}-\alpha')\frac{\phi}{\alpha}-\frac{\gamma'}{\gamma}\phi\Big]
+\frac{\phi^{2}\gamma}{\alpha^{2}n}\Big[\frac{G}{\gamma}+(\alpha-1)|\nabla f|^{2}\Big]^{2}\\
\nonumber
&&-\gamma\phi^{2}\alpha^{2}n^{2}K^{2}-2\phi^{2}\nabla f\nabla G+2h\phi^{2}\gamma(\alpha-1)(l-1)u^{l-1}|\nabla f|^{2}\\
\nonumber
&&+2\big[(\alpha-1)+\alpha(l-1)\big]\phi^{2}\gamma u^{l-1}\nabla f\cdot\nabla h+\phi^{2}\alpha\gamma u^{l-1}\Delta h
-\phi G\sqrt{C}K\\
\nonumber
&\geq&\phi G\Big[\Delta\phi-2\frac{|\nabla\phi|^2}{\phi}+(\frac{2\varphi}{n}-\alpha')\frac{1}{\alpha}-\frac{\gamma'}{\gamma}\Big]
+\frac{\phi^{2}G^{2}}{\alpha^{2}n\gamma}+\frac{\phi^{2}(\alpha-1)^{2}\gamma}{n\alpha^{2}}|\nabla f|^{4}\\
\nonumber
&&
+\frac{2\phi^{2}(\alpha-1)}{n\alpha^{2}}G|\nabla f|^{2}-\gamma\phi^{2}\alpha^{2}n^{2}K^{2}+2\phi G\nabla\phi\nabla f\\
\nonumber
&&+2h\phi^{2}\gamma(\alpha-1)(l-1)u^{l-1}|\nabla f|^{2}+\phi^{2}\alpha\gamma u^{l-1}\Delta h
\\
&&+2\big[(\alpha-1)+\alpha(l-1)\big]\phi^{2}\gamma u^{l-1}\nabla f\cdot\nabla h-\phi G\sqrt{C}K.
\end{eqnarray*}
Using the Cauchy inequality
$$\nabla f\cdot\nabla h\geq -|\nabla f||\nabla h|\geq -h|\nabla f|^{2}-\frac{|\nabla h|^2}{4h},$$
$$\nabla f\cdot\nabla h\leq h|\nabla f|^{2}+\frac{|\nabla h|^2}{4h},$$
we conclude
\begin{align}\label{2.16}
\nonumber
0\geq&\phi G\Big[\Delta\phi-2\frac{|\nabla\phi|^2}{\phi}+(\frac{2\varphi}{n}-\alpha')\frac{\phi}{\alpha}-\frac{\gamma'}{\gamma}\phi\Big]
+\frac{\phi^{2}G^{2}}{\alpha^{2}n\gamma}+\frac{\phi^{2}(\alpha-1)^{2}\gamma}{n\alpha^{2}}|\nabla f|^{4}\\
\nonumber
&
+\frac{2\phi^{2}(\alpha-1)}{n\alpha^{2}}G|\nabla f|^{2}-\gamma\phi^{2}\alpha^{2}n^{2}K^{2}+2\phi G\nabla\phi\nabla f\\
\nonumber
&-2h\phi^{2}\gamma(\alpha-1)(1-l)u^{l-1}|\nabla f|^{2}-2\big[(\alpha-1)+\alpha(1-l)\big]\phi^{2}\gamma u^{l-1}h|\nabla f|^{2}\\
\nonumber
&-\frac{1}{2}\big[(\alpha-1)+\alpha(1-l)\big]\phi^{2}\gamma u^{l-1}\frac{|\nabla h|^2}{h}+\phi^{2}\alpha\gamma u^{l-1}\Delta h-\phi G\sqrt{C}K\\
\nonumber
\geq&\phi G\Big[\Delta\phi-2\frac{|\nabla\phi|^2}{\phi}+(\frac{2\varphi}{n}-\alpha')\frac{\phi}{\alpha}-\frac{\gamma'}{\gamma}\phi\Big]
+\frac{\phi^{2}G^{2}}{\alpha^{2}n\gamma}+\frac{\phi^{2}(\alpha-1)^{2}\gamma}{n\alpha^{2}}|\nabla f|^{4}\\
\nonumber
&
+\frac{2\phi^{2}(\alpha-1)}{n\alpha^{2}}G|\nabla f|^{2}-\gamma\phi^{2}\alpha^{2}n^{2}K^{2}+2\phi G\nabla\phi\nabla f\\
\nonumber
&-2\big[\alpha(3-2\alpha)-1\big]\phi^{2}\gamma u^{l-1}h|\nabla f|^{2}-\frac{1}{2}\big[(\alpha-1)+\alpha(1-l)\big]
\phi^{2}\gamma u^{l-1}\frac{|\nabla h|^2}{h}\\
&+\phi^{2}\alpha\gamma u^{l-1}\Delta h-\phi G\sqrt{C}K,
\end{align}
where we use the fact that $(\alpha-1)(l-1)+(\alpha-1)+\alpha(1-l)\leq \alpha(3-2l)-1$.
Further using the inequality $Ax^{2}+Bx\geq -\frac{B^2}{4A}$ with $A>0$, we have
\begin{equation*}
\frac{2\phi^{2}(\alpha-1)}{n\alpha^{2}}G|\nabla f|^{2}+2\phi G\nabla\phi\nabla f
\geq -\frac{n\alpha^{2}}{2(\alpha-1)}\frac{|\nabla\phi|^{2}}{\phi}\phi G,
\end{equation*}
and
\begin{align*}
&\frac{\phi^{2}(\alpha-1)^{2}\gamma}{n\alpha^{2}}|\nabla f|^{4}-2\big[\alpha(3-2\alpha)-1\big]\phi^{2}\gamma u^{l-1}h|\nabla f|^{2}\\
&\geq -\frac{n\alpha^{2}\big[\alpha(3-2\alpha)-1\big]^2}{(\alpha-1)^2}\gamma\phi^{2}u^{2(l-1)}h^{2}\\
&\geq -\frac{n\alpha^{2}\big[\alpha(3-2\alpha)-1\big]^2}{(\alpha-1)^2}\gamma\phi^{2}\overline{u}^{2}_{1}\delta^{2}_{1}.
\end{align*}
Substituting above two inequalities into ~\eqref{2.16}, we deduce that
\begin{eqnarray*}
0&\geq&\phi G\Big[\Delta\phi-2\frac{|\nabla\phi|^2}{\phi}+(\frac{2\varphi}{n}-\alpha')\frac{\phi}{\alpha}-\frac{\gamma'}{\gamma}\phi
-\frac{n\alpha^{2}}{2(\alpha-1)}\frac{|\nabla\phi|^{2}}{\phi}-\sqrt{C}K\Big]\\
&&
+\frac{\phi^{2}G^{2}}{\alpha^{2}n\gamma}
-\gamma\phi^{2}\alpha^{2}n^{2}K^{2}-\frac{n\alpha^{2}\big[\alpha(3-2\alpha)-1\big]^2}
{(\alpha-1)^2}\gamma\phi^{2}\overline{u}^{2}_{1}\delta^{2}_{1}\\
&&
-\frac{1}{2}\big[(\alpha-1)+\alpha(1-l)\big]
\phi^{2}\gamma \overline{u}_{1}\delta_{2}-\phi^{2}\alpha\gamma \overline{u}_{1}\delta_{3}.
\end{eqnarray*}
Applying ~\eqref{2.11}, we infer
\begin{eqnarray}\label{2.17}
\nonumber
0&\geq&\left[-\frac{C}{R^2}(1+\sqrt{k}R)-\frac{2C}{R^2}
+(\frac{2\varphi}{n}-\alpha')\frac{\phi}{\alpha}-\frac{\gamma'}{\gamma}\phi\right.\\
\nonumber
&&\left.
-\frac{n\alpha^{2}}{2(\alpha-1)}\frac{C}{R^2}-\sqrt{C}K\right]\phi G+\frac{\phi^{2}G^{2}}{\alpha^{2}n\gamma}-\gamma\phi^{2}\alpha^{2}n^{2}K^{2}\\
\nonumber
&&
-\frac{n\alpha^{2}\big[\alpha(3-2\alpha)-1\big]^2}{(\alpha-1)^2}\gamma\phi^{2}\overline{u}^{2}_{1}\delta^{2}_{1}\\
&&-\frac{1}{2}\big[(\alpha-1)+\alpha(1-l)\big]
\phi^{2}\gamma \overline{u}_{1}\delta_{2}-\phi^{2}\alpha\gamma \overline{u}_{1}\delta_{3}.
\end{eqnarray}
For the inequality $Ax^{2}-Bx\leq C$, one has $x\leq \frac{B}{A}+\left(\frac{C}{A}\right)^{\frac{1}{2}}$, where $A, B, C>0$.
By using this inequality to ~\eqref{2.17} and then we arrive at
\begin{eqnarray*}
\phi G(x,T_{1})&\leq& (\phi G)(x_{1},t_1)\\
&\leq&\Big\{n\gamma\alpha^{2}\left[\frac{C}{R^2}(1+\sqrt{K}R)
+\frac{n\alpha^{2}}{2(\alpha-1)}\frac{C}{R^2}+\sqrt{C}K\right]\\
&&+n\gamma\alpha^{2}\left[\frac{\gamma'}{\gamma}-(\frac{2\varphi}{n}-\frac{\alpha'}{\alpha})\frac{1}{\alpha}\right]
+ n^{\frac{3}{2}}\gamma\alpha^{2}\phi K\\
&&+\frac{n\alpha^{2}\big[\alpha(3-2\alpha)-1\big]}{\alpha-1}\phi\gamma\overline{u}_{1}\delta_{1}
+\alpha\phi\gamma\sqrt{n\overline{u}_{1}\delta_3}\\
&&+\sqrt{\frac{\big[(\alpha-1)+\alpha(1-l)\big]}{2}}
\alpha\phi\gamma \sqrt{n\overline{u}_{1}\delta_{2}}\Big\}(x_{1},t_1).
\end{eqnarray*}
If $\gamma$ is nondecreasing which satisfies the system
\begin{equation}\label{2.18}
\left\{\aligned
\frac{\gamma'}{\gamma}-(\frac{2\varphi}{n}-\alpha')\frac{1}{\alpha}\leq 0,\\
\frac{\gamma\alpha^4}{\alpha-1}\leq C.
\endaligned\right.
\end{equation}
Recall that $\alpha(t)$ and $\gamma(t)$ are non-decreasing and $t_{1}<T_{1}$. Hence, we have
\begin{eqnarray*}
\phi G(x,T_{1})&\leq& (\phi G)(x_{1},t_1)\\
&\leq& n\gamma(T_{1})\alpha^{2}(T_{1})\left[\frac{C}{R^2}\Big(1+\sqrt{K}R\Big)+K\right]+\frac{n^{2}C}{R^2}\\
&&+ n^{\frac{3}{2}}\gamma(T_{1})\alpha^{2}(T_{1})\phi K+\phi\alpha(T_{1})\gamma(T_{1})\sqrt{n\overline{u}_{1}\delta_3}\\
&&+\frac{n\alpha^{2}(T_1)\big[\alpha(T_1)(3-2\alpha(T_1))-1\big]}{\alpha(T_1)-1}\phi\gamma(T_1)\overline{u}_{1}\delta_{1}\\
&&+\sqrt{\frac{\big[(\alpha(T_1)-1)+\alpha(T_1)(1-l)\big]}{2}}
\alpha(T_1)\phi\gamma(T_1) \sqrt{n\overline{u}_{1}\delta_{2}}.
\end{eqnarray*}
Hence, we have for  $\phi\equiv 1$ on $B_{R,T}$,
\begin{eqnarray*}
F(x,T_{1})&\leq&n\alpha^{2}(T_{1})\left[\frac{C}{R^2}\Big(1+\sqrt{K}R\Big)+CK\right]+\frac{n^{2}C}{R^{2}\gamma(T_{1})}\\
&&+ n^{\frac{3}{2}}\alpha^{2}(T_{1}) K+\alpha(T_{1})\gamma(T_{1})\sqrt{n\overline{u}_{1}\delta_3}\\
&&+\frac{n\alpha^{2}(T_1)\big[\alpha(T_1)(3-2\alpha(T_1))-1\big]}{\alpha(T_1)-1}\overline{u}_{1}\delta_{1}\\
&&+\sqrt{\frac{\big[(\alpha(T_1)-1)+\alpha(T_1)(1-l)\big]}{2}}
\alpha(T_1) \sqrt{n\overline{u}_{1}\delta_{2}}.
\end{eqnarray*}

If $\gamma$ is nondecreasing which satisfies the system
\begin{equation}\label{2.19}
\left\{\aligned
\frac{\gamma'}{\gamma}-(\frac{2\varphi}{n}-\alpha')\frac{1}{\alpha}\leq 0,\\
\frac{\gamma}{\alpha-1}\leq C.
\endaligned\right.
\end{equation}
Recall that $\alpha(t)$ and $\gamma(t)$ are non-decreasing and $t_{1}<T_{1}$. Hence, we have
\begin{eqnarray*}
\nonumber
\phi G(x,T_{1})&\leq& (\phi G)(x_{1},t_1)\\
\nonumber
&\leq& n\gamma(T_{1})\alpha^{2}(T_{1})\left[\frac{C}{R^2}\Big(1+\sqrt{K}R\Big)+\frac{n\alpha^{2}}{\alpha-1}\frac{C}{R^2}+CK\right]\\
&&+ n^{\frac{3}{2}}\gamma(T_{1})\alpha^{2}(T_{1})\phi K+\phi\alpha(T_{1})\gamma(T_{1})\sqrt{n\overline{u}_{1}\delta_3}\\
&&+\frac{n\alpha^{2}(T_1)\big[\alpha(T_1)(3-2\alpha(T_1))-1\big]}{\alpha(T_1)-1}\phi\gamma(T_1)\overline{u}_{1}\delta_{1}\\
&&+\sqrt{\frac{\big[(\alpha(T_1)-1)+\alpha(T_1)(1-l)\big]}{2}}
\alpha(T_1)\phi\gamma(T_1) \sqrt{n\overline{u}_{1}\delta_{2}}.
\end{eqnarray*}
Hence, we have for  $\phi\equiv 1$ on $B_{R,T}$,
\begin{eqnarray*}
\nonumber
F(x,T_{1})&\leq&n\alpha^{2}(T_{1})\left[\frac{C}{R^2}\Big(1+\sqrt{K}R\Big)+K\right]+\frac{n^{2}C\alpha^{4}}{R^{2}\gamma(T)}\\
&&+ n^{\frac{3}{2}}\gamma(T_{1})\alpha^{2}(T_{1}) K+\alpha(T_{1})\sqrt{n\overline{u}_{1}\delta_3}\\
&&+\frac{n\alpha^{2}(T_1)\big[\alpha(T_1)(3-2\alpha(T_1))-1\big]}{\alpha(T_1)-1}\overline{u}_{1}\delta_{1}\\
&&+\sqrt{\frac{\big[(\alpha(T_1)-1)+\alpha(T_1)(1-l)\big]}{2}}
\alpha(T_1) \sqrt{n\overline{u}_{1}\delta_{2}}.
\end{eqnarray*}
Because $T_{1}$ is arbitrary in $0<T_{1}<T$ and
$\alpha(3-2\alpha)-1\leq \alpha-1$ and $\alpha-1+\alpha(1-l)\leq \alpha(2-l)$. Thus  the conclusion is valid.

\textbf{Case $2$}\quad $l>1$.\\
Substituting  ~\eqref{2.14} into ~\eqref{2.10}, we have
 \begin{align*}
(\Delta-\partial_{t}) G
\geq &\gamma|f_{ij}+\frac{\varphi}{n}\delta_{ij}|^{2}+\Big[(\frac{2\varphi}{n}-\alpha')\frac{1}{\alpha}-\frac{\gamma'}{\gamma}\big]G
-\gamma\alpha^{2}n^{2}K^{2}\\
\nonumber
&-2\nabla f\nabla G+2(l\alpha-1)\gamma u^{l-1}\nabla f\cdot\nabla h-h(l-1)u^{l-1}G\\
&-h\gamma(l-1)u^{l-1}\alpha\varphi
+h\gamma(l-1)(l\alpha-1)u^{l-1}|\nabla f|^{2}+\alpha\gamma u^{l-1}\Delta h.
\end{align*}
 Using ~\eqref{2.13}, we infer
\begin{eqnarray}\label{2.20}
\nonumber
0&\geq&(\Delta-\partial_{t}) (\phi G)\\
\nonumber
&=&G\Big(\Delta\phi-2\frac{|\nabla\phi|^2}{\phi}\Big)+\phi(\Delta-\partial_{t})G-\gamma G\phi_{t}\\
\nonumber
&\geq&G\Big(\Delta\phi-2\frac{|\nabla\phi|^2}{\phi}\Big)+\frac{\phi\gamma}{\alpha^{2}n}\Big[\frac{G}{\gamma}+(\alpha-1)|\nabla f|^{2}\Big]^{2}\\
\nonumber
&&+\Big[(\frac{2\varphi}{2}-\alpha')\frac{1}{\alpha}-\frac{\gamma'}{\gamma}\Big]\phi G-\gamma\phi\alpha^{2}n^{2}K^{2}-2\phi\nabla f\nabla G\\
\nonumber
&&+2(l\alpha-1)\phi\gamma u^{l-1}\nabla f\cdot\nabla h
-h(l-1)u^{l-1}G-h\gamma(l-1)u^{l-1}\phi\alpha\varphi\\
&&+\phi\alpha\gamma u^{l-1}\Delta h+h\phi\gamma(l-1)(l\alpha-1)u^{l-1}|\nabla f|^{2}-G\sqrt{C}K.
\end{eqnarray}
Multiply $\phi$ to (2.20), we have
\begin{eqnarray}\label{2.21}
\nonumber
0&\geq&\phi G\Big[\Delta\phi-2\frac{|\nabla\phi|^2}{\phi}+(\frac{2\varphi}{n}-\alpha')\frac{\phi}{\alpha}-\frac{\gamma'}{\gamma}\phi\Big]
+\frac{\phi^{2}\gamma}{\alpha^{2}n}\Big[\frac{G}{\gamma}+(\alpha-1)|\nabla f|^{2}\Big]^{2}\\
\nonumber
&&-\gamma\phi^{2}\alpha^{2}n^{2}K^{2}-2\phi^{2}\nabla f\nabla G+2(l\alpha-1)\phi^{2}\gamma u^{l-1}\nabla f\cdot\nabla h\\
\nonumber
&&+\phi^{2}\alpha\gamma u^{l-1}\Delta h
-h(l-1)u^{l-1}\phi G-h\gamma(l-1)u^{l-1}\phi^{2}\alpha\varphi\\
\nonumber
&&+h\phi^{2}\gamma(l-1)(l\alpha-1)u^{l-1}|\nabla f|^{2}-\phi G\sqrt{C}K\\
\nonumber
&\geq&\phi G\Big[\Delta\phi-2\frac{|\nabla\phi|^2}{\phi}+(\frac{2\varphi}{n}-\alpha')\frac{\phi}{\alpha}-\frac{\gamma'}{\gamma}\phi\Big]
+\frac{\phi^{2}G^{2}}{\alpha^{2}n\gamma}\\
\nonumber
&&
+\frac{2\phi^{2}(\alpha-1)}{n\alpha^{2}}G|\nabla f|^{2}-\gamma\phi^{2}\alpha^{2}n^{2}K^{2}+2\phi G\nabla\phi\nabla f\\
\nonumber
&&+2(l\alpha-1)\phi^{2}\gamma u^{l-1}\nabla f\cdot\nabla h+\phi^{2}\alpha\gamma u^{l-1}\Delta h\\
\nonumber
&&-h(l-1)u^{l-1}\phi G
-h\gamma(l-1)u^{l-1}\phi^{2}\alpha\varphi\\
&&
+h\phi^{2}\gamma(l-1)(l\alpha-1)u^{l-1}|\nabla f|^{2}-\phi G\sqrt{C}K,
\end{eqnarray}
where we drop one term $\frac{\phi^{2}(\alpha-1)^{2}\gamma}{n\alpha^{2}}|\nabla f|^{4}$.\\

Further using the inequality $Ax^{2}+Bx\geq -\frac{B^2}{4A}$ with $A>0$, we have
\begin{equation*}
\frac{2\phi^{2}(\alpha-1)}{n\alpha^{2}}G|\nabla f|^{2}+2\phi G\nabla\phi\nabla f
\geq -\frac{n\alpha^{2}}{2(\alpha-1)}\frac{|\nabla\phi|^{2}}{\phi}\phi G,
\end{equation*}
and
\begin{align*}
&h\phi^{2}\gamma(l-1)(l\alpha-1)u^{l-1}|\nabla f|^{2}+2(l\alpha-1)\phi^{2}\gamma u^{l-1}\nabla f\cdot\nabla h\\
&\geq h\phi^{2}\gamma(l-1)(l\alpha-1)u^{l-1}|\nabla f|^{2}-2(l\alpha-1)\phi^{2}\gamma u^{l-1}|\nabla f|\cdot|\nabla h|\\
&\geq -\frac{l\alpha-1}{l-1}\gamma\phi^{2}u^{l-1}\frac{|\nabla h|^{2}}{h}\\
&\geq -\frac{l\alpha-1}{l-1}\gamma\phi^{2}\overline{u}_{2}\delta_{2}.
\end{align*}
Substituting above two inequalities into (2.21), we deduce that
\begin{eqnarray*}
0&\geq&\phi G\Big[\Delta\phi-2\frac{|\nabla\phi|^2}{\phi}+(\frac{2\varphi}{n}-\alpha')\frac{\phi}{\alpha}-\frac{\gamma'}{\gamma}\phi
-\frac{n\alpha^{2}}{2(\alpha-1)}\frac{|\nabla\phi|^{2}}{\phi}\\
&&-\delta_{1}(l-1)\overline{u}_{2}-\sqrt{C}K\Big]+\frac{\phi^{2}G^{2}}{\alpha^{2}n\gamma}
-\gamma\phi^{2}\alpha^{2}n^{2}K^{2}\\
&&
-\frac{l\alpha-1}{l-1}\gamma\phi^{2}\overline{u}_{2}\delta_{2}
-\gamma(l-1)\overline{u}_{2}\phi^{2}\delta_{1}\alpha\varphi
-\phi^{2}\alpha\gamma \overline{u}_{2}\delta_{3}.
\end{eqnarray*}
Applying ~\eqref{2.11}, we infer
\begin{eqnarray}\label{2.22}
\nonumber
0&\geq&\left[-\frac{C}{R^2}(1+\sqrt{k}R)-\frac{2C}{R^2}
+(\frac{2\varphi}{n}-\alpha')\frac{\phi}{\alpha}-\frac{\gamma'}{\gamma}\phi
-\frac{n\alpha^{2}}{2(\alpha-1)}\frac{C}{R^2}\right.\\
\nonumber
&&\left.-\delta_{1}(l-1)\overline{u}_{2}-\sqrt{C}K\right]\phi G+\frac{\phi^{2}G^{2}}{\alpha^{2}n\gamma}
-\gamma\phi^{2}\alpha^{2}n^{2}K^{2}\\
&&
-\frac{l\alpha-1}{l-1}\gamma\phi^{2}\overline{u}_{2}\delta_{2}
-\gamma(l-1)\overline{u}_{2}\phi^{2}\delta_{1}\alpha\varphi
-\phi^{2}\alpha\gamma \overline{u}_{2}\delta_{3}.
\end{eqnarray}
For the inequality $Ax^{2}-2Bx\leq C$, one has $x\leq \frac{2B}{A}+\left(\frac{C}{A}\right)^{\frac{1}{2}}$, where $A, B, C>0$.
By using this inequality to ~\eqref{2.22} and then we arrive at
\begin{eqnarray*}
\nonumber
\phi G(x,T_{1})&\leq& (\phi G)(x_{1},t_1)\\
\nonumber
&\leq& \left\{n\gamma\alpha^{2}\left[\frac{C}{R^2}(1+\sqrt{K}R)+\frac{n\alpha^{2}}{2(\alpha-1)}\frac{C}{R^2}
+\delta_{1}(l-1)\overline{u}_{2}+\sqrt{C}K\right]\right.\\
&&+n\gamma\alpha^{2}\left[\frac{\gamma'}{\gamma}-(\frac{2\varphi}{n}-\frac{\alpha'}{\alpha})\frac{1}{\alpha}\right]
+ n^{\frac{3}{2}}\gamma\phi\alpha^{2} K+\alpha\phi\gamma\sqrt{\frac{n(l\alpha-1)\overline{u}_{2}\delta_{2}}{l-1}}\\
&&\left.
+\alpha^{\frac{3}{2}}\gamma\phi\sqrt{n(l-1)\delta_{1}\varphi}+\alpha^{\frac{3}{2}}\gamma\phi\sqrt{n\delta_{3}\overline{u}_{2}}
\right\}(x_{1},t_1).
\end{eqnarray*}
If $\gamma$ is nondecreasing which satisfies the system
\begin{equation}\label{2.23}
\left\{\aligned
\frac{\gamma'}{\gamma}-(\frac{2\varphi}{n}-\alpha')\frac{1}{\alpha}\leq 0,\\
\frac{\gamma\alpha^4}{\alpha-1}\leq C.
\endaligned\right.
\end{equation}
Recall that $\alpha(t)$ and $\gamma(t)$ are non-decreasing and $t_{1}<T_{1}$. Hence, we have
\begin{eqnarray*}
\nonumber
\phi G(x,T_{1})&\leq& (\phi G)(x_{1},t_1)\\
\nonumber
&\leq& n\gamma(T_{1})\alpha^{2}(T_{1})\left[\frac{C}{R^2}\Big(1+\sqrt{K}R\Big)+K\right]+\frac{n^{2}C}{R^2}
\\
&&+n\gamma(T_{1})\alpha^{2}(T_{1})(l-1)\delta_{1}\overline{u}_{2}+ n^{\frac{3}{2}}\gamma\phi(T_{1})\alpha^{2}(T_{1}) K\\
\nonumber
&&
+\alpha(T_1)\phi\gamma(T_1)\sqrt{\frac{n(l\alpha(T_1)-1)\overline{u}_{2}\delta_{2}}{l-1}}\\
&&
+\alpha^{\frac{3}{2}}(T_{1})\gamma(T_{1})\phi\sqrt{n(l-1)\delta_{1}\varphi}
+\alpha^{\frac{3}{2}}(T_{1})\gamma(T_{1})\phi\sqrt{n\delta_{3}\overline{u}_{2}}.
\end{eqnarray*}
Hence, we have for  $\phi\equiv 1$ on $B_{R,T}$,
\begin{eqnarray*}
F(x,T_{1})&\leq&n\alpha^{2}(T_{1})\left[\frac{C}{R^2}\Big(1+\sqrt{K}R\Big)+K\right]+\frac{n^{2}C}{R^2}
\\
&&+n\alpha^{2}(T_{1})(l-1)\delta_{1}\overline{u}_{2}+ n^{\frac{3}{2}}\alpha^{2}(T_{1}) K\\
&&+
\alpha(T_1)\sqrt{\frac{n(l\alpha(T_1)-1)\overline{u}_{2}\delta_{2}}{l-1}}\\
&&
+\alpha^{\frac{3}{2}}(T_{1})\sqrt{n(l-1)\delta_{1}\varphi}+\alpha^{\frac{3}{2}}(T_{1})\sqrt{n\delta_{3}\overline{u}_{2}}.
\end{eqnarray*}

If $\gamma$ is nondecreasing which satisfies the system
\begin{equation}\label{2.24}
\left\{\aligned
\frac{\gamma'}{\gamma}-(\frac{2\varphi}{n}-\alpha')\frac{1}{\alpha}\leq 0,\\
\frac{\gamma}{\alpha-1}\leq C.
\endaligned\right.
\end{equation}
Recall that $\alpha(t)$ and $\gamma(t)$ are non-decreasing and $t_{1}<T_{1}$. Hence, we have
\begin{eqnarray*}
\nonumber
\phi G(x,T_{1})&\leq& (\phi G)(x_{1},t_1)\\
\nonumber
&\leq& n\gamma(T_{1})\alpha^{2}(T_{1})\left[\frac{C}{R^2}\Big(1+\sqrt{K}R\Big)+\frac{C n\alpha^{2}(T)}{R^2}+CK\right]\\
&&+n\gamma(T_{1})\alpha^{2}(T_{1})(l-1)\delta_{1}\overline{u}_{2}+ n^{\frac{3}{2}}\gamma(T_{1})\phi\alpha^{2}(T_{1}) K\\
&&+
\alpha(T_1)\phi\gamma(T_1)\sqrt{\frac{n(l\alpha(T_1)-1)\overline{u}_{2}\delta_{2}}{l-1}}\\
&&+\alpha^{\frac{3}{2}}(T_{1})\gamma(T_{1})\phi\sqrt{n(l-1)\delta_{1}\varphi}
+\alpha^{\frac{3}{2}}(T_{1})\gamma(T_{1})\phi\sqrt{n\delta_{3}\overline{u}_{2}}.
\end{eqnarray*}
Hence, we have for  $\phi\equiv 1$ on $B_{R,T}$,
\begin{eqnarray*}
F(x,T_{1})&\leq&n\alpha^{2}(T_{1})\left[\frac{C}{R^2}\Big(1+\sqrt{K}R\Big)+\frac{C n\alpha^{2}(T)}{R^2}+CK\right]\\
&&+n\alpha^{2}(T_{1})(l-1)\delta_{1}\overline{u}_{2}+ n^{\frac{3}{2}}\alpha^{2}(T_{1}) K\\
&&+
\alpha(T_1)\sqrt{\frac{n(l\alpha(T_1)-1)\overline{u}_{2}\delta_{2}}{l-1}}\\
&&
+\alpha^{\frac{3}{2}}(T_{1})\sqrt{n(l-1)\delta_{1}\varphi}
+\alpha^{\frac{3}{2}}(T_{1})\sqrt{n\delta_{3}\overline{u}_{2}}.
\end{eqnarray*}
Because $T_{1}$ is arbitrary in $0<T_{1}<T$,  the conclusion is valid.

\end{proof}

\begin{proof}[\textbf{Proof of Theorem $2.2$}]
Since $u_{t}=\nabla u+h(x,t)u^{l}$, we have
\begin{eqnarray*}
\partial_{t}(|\nabla u|^{2})&=&2\mathrm{Ric}(\nabla u, \nabla u)+2<\nabla u, \nabla(u_t)>\\
&=&2\mathrm{Ric}(\nabla u, \nabla u)+2<\nabla u, \nabla(\Delta u)>+2<\nabla u,\nabla(h(x,t)u^l)>.
\end{eqnarray*}
Applying Bochner's formula, above equation becomes
\begin{eqnarray}\label{2.20}
\partial_{t}(|\nabla u|^{2})&=&\Delta (|\nabla u|^2)-2|\nabla^{2}u|^{2}+2<\nabla u,\nabla(h(x,t)u^l)>.
\end{eqnarray}
Besides,
\begin{eqnarray}\label{2.21}
\partial_{t}(u^{2})&=&\Delta(u^2)-2|\nabla u|^{2}+2h(t)u^{l+1}.
\end{eqnarray}
Let $\overline{F}=t|\nabla u|^{2}+Xu^{2}$, where $X$ is a constant to be decided.
Then combining ~\eqref{2.20} with ~\eqref{2.21}, we obtain
\begin{eqnarray}\label{2.21}
\nonumber
\partial_{t}\overline{F}&=&|\nabla u|^{2}+t[\Delta (|\nabla u|^2)-2|\nabla^{2}u|^{2}+2<\nabla u,\nabla(h(t)u^l)>]\\
\nonumber
&&+X[\Delta(u^2)-2|\nabla u|^{2}+2h(t)u^{l+1}]\\
\nonumber
&=&|\nabla u|^{2}+t[\Delta (|\nabla u|^2)-2|\nabla^{2}u|^{2}+2h(t)(l-1)u^{l-2}|\nabla u|^{2}\\
\nonumber
&&+X[\Delta(u^2)-2|\nabla u|^{2}+2h(t)u^{l+1}]\\
&\leq &\Delta\overline{F}+(1-2X)|\nabla u|^{2}.
\end{eqnarray}
Selecting $X=\frac{1}{2}$ and using maximum principle, we infer
$$\overline{F}(x,t)\leq \max_{x\in M^n}\overline{F}(x,0)=\frac{1}{2}\max_{x\in M^n}u^{2}(x,0),$$
which implies the theorem is valid.
\end{proof}

\section{\textbf{Gradient estimates for the equation $(1.6)$}}

Recalled that $(M^{n}, g(t))$ is called a gradient Ricci soliton if there is a smooth function $f$ on $M^n$ such that for some constant $c\in \mathbb{R}$, which satisfies
\begin{equation}\label{3.1}
Rc=cg+D^{2}f,
\end{equation}
where $D^{2}f$ is the Hessian of $f$. Let $u=e^{f}$, after some computation applying (3.1) as done in [21], we get
$$\Delta u+2cu\log u=(A_{0}-nc)u\quad in\quad M^n$$
for some constant $A_0$, where $n$ is the dimension of $M^n$.
In [21], Ma proved the local gradient estimate of positive solutions to the
equation
$$\Delta u+au\log u+bu=0\quad in\quad M^{n},$$
where $a>0$ and $b\in \mathbb{R}$  are constants for complete noncaompact manifolds with a fixed metric and curvature locally bounded below. In [31], Yang generalized Ma's result and derived a local gradient estimates for positive solutions to the equation
$$u_{t}=\Delta u+au\log u+bu\quad in\quad M^{n}\times(0,T],$$
where $a,b\in\mathbb{R}$ are constants for complete noncompact manifolds with a fixed metric and
curvature locally bounded below. Replacing $u$ by $e^{b/a}u$, above equation becomes
\begin{equation}\label{3.2}
u_{t}=\Delta u +au\log u.
\end{equation}
One can find in \cite{29,30} some related results for equation ~\eqref{3.2} on manifolds.

In this section,  we consider the nonlinear parabolic equation ~\eqref{1.6}
under the Ricci flow.

\subsection{\textbf{Main results}}
\

Our main results state as follows.

\begin{thm}
Let $(M^{n}, g(t))_{t\in [0,T]}$ be a complete solution to the Ricci flow ~\eqref{1.7}.
  Assume that $|\mathrm{Ric}(x,t)|\leq K$ for some $K>0$ and all $t\in [0,T]$.
  Suppose that there exist three functions $\alpha(t)$, $\varphi(t)$ and $\gamma(t)$
  satisfy the following conditions (C1), (C2), (C3) and (C4).

Given $x_{0}\in M$ and $R>0$, let $u$ be a positive solution of the nonlinear parabolic equation
$$\partial_{t}u=\Delta u+au\log u$$
in the cube $B_{2R,T}:=\{(x,t)|d(x,x_{0},t)\leq 2R, 0\leq t\leq T\}$, where $a$ is a constant.

(1)\quad For $a\leq 0$.
If $\frac{\gamma\alpha^4}{\alpha-1}\leq C_{1}$ for some constant $C_1$,
then
\begin{eqnarray*}
\frac{|\nabla u|^2}{u^2}-\alpha\frac{u_{t}}{u}+\alpha a\log u
&\leq&
C\alpha^{2}\left(\frac{1}{R^2}+\frac{\sqrt{K}}{R}+K\right)\\
&&+\frac{Cn^2}{R^{2}\gamma}
+n^{\frac{3}{2}}\alpha^{2}K+n|a|\alpha^{2}+\alpha\varphi.
\end{eqnarray*}

If $\frac{\gamma}{\alpha-1}\leq C_{2}$ for some constant $C_2$,
then
\begin{eqnarray*}
\frac{|\nabla u|^2}{u^2}-\alpha\frac{u_{t}}{u}+\alpha a\log u
&\leq&
C\alpha^{2}\left(\frac{1}{R^2}+\frac{\sqrt{K}}{R}+K\right)\\
&&+\frac{Cn^{2}\alpha^{4}}{R^{2}\gamma}
+n^{\frac{3}{2}}\alpha^{2}K+n|a|\alpha^{2}+\alpha\varphi.
\end{eqnarray*}

(2)\quad For $a>0$.
If $\frac{\gamma\alpha^4}{\alpha-1}\leq C_{1}$ for some constant $C_1$,
then
\begin{eqnarray*}
\frac{|\nabla u|^2}{u^2}-\alpha\frac{u_{t}}{u}+\alpha a\log u
&\leq&
C\alpha^{2}\left(\frac{1}{R^2}+\frac{\sqrt{K}}{R}+a+K\right)+\frac{n^{2}C}{R^{2}\gamma}\\
&&+ n^{\frac{3}{2}}\alpha^{2} K+\alpha\varphi.
\end{eqnarray*}

If $\frac{\gamma}{\alpha-1}\leq C_{2}$ for some constant $C_2$,
then
\begin{eqnarray*}
\frac{|\nabla u|^2}{u^2}-\alpha\frac{u_{t}}{u}+\alpha a\log u
&\leq&
C\alpha^{2}\left(\frac{1}{R^2}+\frac{\sqrt{K}}{R}+a+K\right)+\frac{n^{2}C\alpha^4}{R^{2}\gamma}\\
&&+ n^{\frac{3}{2}}\alpha^{2} K+\alpha\varphi.
\end{eqnarray*}
where $C$ is a constant.
\end{thm}

\begin{cor}  Let $(M^{n}, g(t))_{t\in [0,T]}$ be a complete solution to the Ricci flow ~\eqref{1.7}.
  Assume that $|\mathrm{Ric}(x,t)|\leq K$ for some $K>0$ and all $t\in [0,T]$.
Given $x_{0}\in M$ and $R>0$, let $u$ be a positive solution of the nonlinear parabolic equation  ~\eqref{1.6}
in the cube $B_{2R,T}:=\{(x,t)|d(x,x_{0},t)\leq 2R, 0\leq t\leq T$\}. Then the following special estimates are valid.

1. Li-Yau type:
$$\alpha(t)=constant,\quad \varphi(t)=\frac{\alpha n}{t}+\frac{nK\alpha^2}{\alpha-1}, \gamma(t)=t^{\theta}
\quad  with \quad 0<\theta\leq 2.$$
If $a\leq 0$, then
\begin{align*}
\frac{|\nabla u|^2}{u^2}-\alpha\frac{u_{t}}{u}+\alpha a\log u
\leq&
C\alpha^{2}\left(\frac{1}{R^2}+\frac{\sqrt{K}}{R}+\frac{\alpha^2}{\alpha-1}\frac{1}{R^2}+K\right)\\
&+n^{\frac{3}{2}}\alpha^{2}K+n|a|\alpha^{2}+\alpha\varphi.
\end{align*}
If $a>0$, then
\begin{align*}
\frac{|\nabla u|^2}{u^2}-\alpha\frac{u_{t}}{u}+\alpha a\log u
\leq&
C\alpha^{2}\left(\frac{1}{R^2}+\frac{\sqrt{K}}{R}+\frac{\alpha^2}{\alpha-1}\frac{1}{R^2}+a+K\right)\\
&+n^{\frac{3}{2}}\alpha^{2}K+\alpha\varphi.
\end{align*}
2. Hamilton type:
$$\alpha(t)=e^{2Kt},\quad \varphi(t)=\frac{n}{t}e^{4Kt},\quad \gamma(t)=te^{2Kt}.$$
If $a\leq 0$, then
\begin{align*}
\frac{|\nabla u|^2}{u^2}-\alpha\frac{u_{t}}{u}+\alpha a\log u
&\leq
C\alpha^{2}\left(\frac{1}{R^2}+\frac{\sqrt{K}}{R}+K\right)\\
&+\frac{C\alpha^4}{R^{2}te^{2Kt}}+\alpha\varphi
+n^{\frac{3}{2}}\alpha^{2}K+n|a|\alpha^{2}+\alpha\varphi.
\end{align*}
If $a>0$, then
\begin{align*}
\frac{|\nabla u|^2}{u^2}-\alpha\frac{u_{t}}{u}+\alpha a\log u
&\leq
C\alpha^{2}\left(\frac{1}{R^2}+\frac{\sqrt{K}}{R}+a+K\right)\\
&+\frac{C\alpha^4}{R^{2}te^{2Kt}}+n^{\frac{3}{2}}\alpha^{2}K+\alpha\varphi.
\end{align*}
3. Li-Xu type:
\begin{align*}&\alpha(t)=1+\frac{\sinh(Kt)\cosh(Kt)-Kt}{\sinh^{2}(Kt)},\quad \varphi(t)=2nK[1+\coth(Kt)],\\
& \gamma(t)=\tanh(Kt).
\end{align*}
If $a\leq 0$, then
\begin{align*}
\frac{|\nabla u|^2}{u^2}-\alpha\frac{u_{t}}{u}+\alpha a\log u
&\leq
C\alpha^{2}\left(\frac{1}{R^2}+\frac{\sqrt{K}}{R}+K\right)\\
&+\frac{C}{R^{2}\tanh(Kt)}+n^{\frac{3}{2}}\alpha^{2}K+n|a|\alpha^{2}+\alpha\varphi.
\end{align*}
If $a>0$, then
\begin{align*}
\frac{|\nabla u|^2}{u^2}-\alpha\frac{u_{t}}{u}+\alpha a\log u
\leq&
C\alpha^{2}\left(\frac{1}{R^2}+\frac{\sqrt{K}}{R}+a+K\right)\\
&+\frac{C}{R^{2}\tanh(Kt)}+n^{\frac{3}{2}}\alpha^{2}K+\alpha\varphi.
\end{align*}
4. Linear Li-Xu type:
$$\alpha(t)=1+2Kt, \varphi(t)=\frac{n}{t}+nK(1+2Kt+\mu Kt),\gamma(t)=Kt \quad with \quad \mu\geq \frac{1}{4}.$$
If $a\leq 0$, then
\begin{align*}
\frac{|\nabla u|^2}{u^2}-\alpha\frac{u_{t}}{u}+\alpha a\log u
&\leq
C\alpha^{2}\left(\frac{1}{R^2}+\frac{\sqrt{K}}{R}+K\right)\\
&+\frac{C\alpha^4}{R^{2}Kt}+n^{\frac{3}{2}}\alpha^{2}K+n|a|\alpha^{2}+\alpha\varphi.
\end{align*}
If $a>0$, then
\begin{align*}
\frac{|\nabla u|^2}{u^2}-\alpha\frac{u_{t}}{u}+\alpha a\log u
&\leq
C\alpha^{2}\left(\frac{1}{R^2}+\frac{\sqrt{K}}{R}+a+K\right)\\
&+\frac{C\alpha^4}{R^{2}Kt}+n^{\frac{3}{2}}\alpha^{2}K+\alpha\varphi.
\end{align*}
\end{cor}

The local estimates above imply global estimates.
\begin{cor}
Let $(M^{n}, g(0))$ be a complete noncompact Riemannian manifold without
boundary, and asuume $g(t)$ evolves by Ricci flow in such a way that $|\mathrm{Ric}|\leq K$ for $t\in[0,T]$.
Let $u(x,t)$ be a positive solution to the equation ~\eqref{1.6}.
If $l\in \mathbb{R}$ and for $(x,t)\in M^{n}\times (0,T]$, then
\begin{align*}
\nonumber
\frac{|\nabla u|^2}{u^2}-\alpha\frac{u_{t}}{u}+\alpha a\log u
\leq&
C\alpha^{2}(K+|a|)+\alpha\varphi.
\end{align*}
\end{cor}

\begin{rem}
The above results may be regard as generalizing the gradient
estimates of  Yang \cite{30} to the Ricci flow.
\end{rem}
\subsection{\textbf{Auxiliary lemma}}
\

To prove the theorem $3.1$, the following a lemma is needed.

Let $f=\log u$. Then
\begin{equation}\label{3.3}
(\Delta-\partial_{t})f=-|\nabla f|^{2}+af.
\end{equation}
Let $F=|\nabla f|^{2}-\alpha f_{t}+\alpha af-\alpha\varphi$, where $\alpha=\alpha(t)$ and $\varphi=\varphi(t)$.
Then
\begin{eqnarray}\label{3.4}
\nonumber
\Delta f&=& f_{t}-af-|\nabla f|^{2}\\
&=&-\frac{F}{\alpha}-(\frac{\alpha-1}{\alpha})|\nabla f|^{2}-\varphi.
\end{eqnarray}

\begin{lem}
We assume that $\alpha(t)>1$ and $\varphi(t)>0$ satisfy the following system
(2.3). Then
\begin{eqnarray}\label{3.5}
\nonumber
(\Delta-\partial_{t}) F&\geq &|f_{ij}+\frac{\varphi}{n}\delta_{ij}|^{2}+(\frac{2\varphi}{n}-\alpha')\frac{1}{\alpha}F
-\alpha^{2}n^{2}K^{2}-2\nabla f\nabla F\\
&&+2a(\alpha-1)|\nabla f|^{2}+a\alpha\Delta f.
\end{eqnarray}
\end{lem}

\begin{proof}[\textbf{Proof}]
A computation is shown that
\begin{eqnarray}\label{3.6}
\nonumber
\Delta F&=&\Delta|\nabla f|^{2}-\alpha\Delta(f_{t})+\alpha a\Delta f\\
\nonumber
&=&2|f_{ij}|^{2}+2f_{j}f_{iij}+2R_{ij}f_{i}f_{j}-\alpha\Delta(f_t)+\alpha a\Delta f\\
\nonumber
&=&2\Big(|f_{ij}|^{2}+\alpha R_{ij}f_{ij}\Big)+2f_{j}f_{iij}+2R_{ij}f_{i}f_{j}
-\alpha(\Delta f)_{t}+\alpha a \Delta f\\
\nonumber
&\geq& 2|f_{ij}|^{2}-2\sum \alpha |R_{ij}||f_{ij}|+2f_{j}f_{iij}+2R_{ij}f_{i}f_{j}
-\alpha(\Delta f)_{t}+\alpha a \Delta f\\
\nonumber
&\geq&2|f_{ij}|^{2}-\sum (\alpha^{2} |R_{ij}|^{2}+|f_{ij}|^{2})+2f_{j}f_{iij}+2R_{ij}f_{i}f_{j}
-\alpha(\Delta f)_{t}+\alpha a \Delta f\\
\nonumber
&\geq&|f_{ij}|^{2}-\sum\alpha^{2} |R_{ij}|^{2}+2f_{j}f_{iij}+2R_{ij}f_{i}f_{j}
-\alpha(\Delta f)_{t}+\alpha a \Delta f\\
&\geq&|f_{ij}|^{2}-\alpha^{2}n^{2}K^{2}+2f_{j}f_{iij}+2R_{ij}f_{i}f_{j}
-\alpha(\Delta f)_{t}+\alpha a \Delta f.
\end{eqnarray}
and
\begin{eqnarray}\label{3.7}
\nonumber
\partial_{t} F&=&(|\nabla f|^{2})_{t}-\alpha f_{tt}-\alpha' f_{t}+\alpha' af+\alpha af_{t}-\alpha \varphi'-\alpha'\varphi\\
\nonumber
&=&2\nabla f\nabla(f_t)+2R_{ij} f_{i}f_{j}-\alpha f_{tt}-\alpha'f_{t}+\alpha'a f\\
&&+\alpha af_{t}-\alpha \varphi'-\alpha'\varphi.
\end{eqnarray}
We follow that from ~\eqref{3.6} and ~\eqref{3.7}
\begin{eqnarray}\label{3.8}
\nonumber
(\Delta-\partial_{t}) F&\geq &|f_{ij}|^{2}-\alpha^{2}n^{2}K^{2}+2\nabla f\nabla(\Delta f)
-\alpha(\Delta f)_{t}+\alpha a \Delta d-2\nabla f\nabla(f_t)\\
\nonumber
&&+\alpha f_{tt}+\alpha'f_{t}-\alpha' af-\alpha af_{t}+\alpha \varphi'+\alpha'\varphi\\
\nonumber
&=&|f_{ij}|^{2}-\alpha^{2}n^{2}K^{2}+2\nabla f\nabla(\Delta f)
-\alpha(f_{t}-|\nabla f|^{2}-af)_{t}\\
\nonumber
&&+\alpha a \Delta f-2\nabla f\nabla(f_t)+\alpha f_{tt}+\alpha'f_{t}\\
\nonumber
&&-\alpha' af-\alpha c f_{t}+\alpha \varphi'+\alpha'\varphi\\
\nonumber
&=&|f_{ij}|^{2}-\alpha^{2}n^{2}K^{2}+2\nabla f\nabla(\Delta f)
+\alpha(|\nabla f|^{2})_{t}+\alpha a \Delta f\\
\nonumber
&&-2\nabla f\nabla(f_t)+\alpha'f_{t}-\alpha'af+\alpha \varphi'+\alpha'\varphi\\
\nonumber
&=&|f_{ij}|^{2}-\alpha^{2}n^{2}K^{2}+2\nabla f\nabla(\Delta f)
+2\alpha\nabla f\nabla(f_t)+2\alpha R_{ij}f_{i}f_{j}\\
\nonumber
&&+\alpha a \Delta f-2\nabla f\nabla(f_t)+\alpha'f_{t}-\alpha' af+\alpha \varphi'+\alpha'\varphi\\
\nonumber
&=&|f_{ij}|^{2}+2\alpha R_{ij}f_{i}f_{j}-\alpha^{2}n^{2}K^{2}+2\nabla f\nabla F
+2a(\alpha-1)(l-1)|\nabla f|^{2}\\
&&+\alpha a \Delta f+\alpha'f_{t}-\alpha' a f+\alpha \varphi'+\alpha'\varphi.
\end{eqnarray}
Further, by utilizing the unit matrix $(\delta_{ij})_{n\times n}$ and $(3.8)$, we obtain
\begin{eqnarray}\label{3.9}
\nonumber
(\Delta-\partial_{t}) F&\geq &|f_{ij}+\frac{\varphi}{n}\delta_{ij}|^{2}-2\alpha K|\nabla f|^{2}-\alpha^{2}n^{2}K^{2}+2\nabla f\nabla F
\\
\nonumber
&&+2a(\alpha-1)|\nabla f|^{2}+a\alpha \Delta f+\alpha'f_{t}-\alpha' af+\alpha \varphi'+\alpha'\varphi\\
\nonumber
&&-\frac{\varphi^2}{n}-2\frac{\varphi}{n}\Delta f\\
\nonumber
&= &|f_{ij}+\frac{\varphi}{n}\delta_{ij}|^{2}+(\frac{2\varphi}{n}-2\alpha K)|\nabla f|^{2}-(\frac{2\varphi}{n}-\alpha')f_{t}\\
\nonumber
&&+(\frac{2\varphi}{n}-\alpha')cu^{l-1}-\alpha^{2}n^{2}K^{2}-2\nabla f\nabla F+2a(\alpha-1)|\nabla f|^{2}\\
&&
+a\alpha\Delta f+\alpha \varphi'+\alpha'\varphi-\frac{\varphi^2}{n}+(\frac{2\varphi}{n}-\alpha')\frac{\alpha\varphi}{\alpha}.
\end{eqnarray}

\end{proof}

\subsection{\textbf{The proof of Theorem}}
\

In this section, we will prove  Theorem $3.1$.
\begin{proof}[\textbf{Proof of Theorem $3.1$}]
Let $G=\gamma(t)F$ and $\gamma(t)>0$ be non-decreasing. Then
\begin{eqnarray}\label{3.10}
\nonumber
(\Delta-\partial_{t}) G&=&\gamma(\Delta-\partial_{t})F-\gamma'F\\
\nonumber
&\geq &\gamma|f_{ij}+\frac{\varphi}{n}g_{ij}|^{2}+(\frac{2\varphi}{n}-\alpha')\frac{1}{\alpha}G
-\gamma\alpha^{2}n^{2}K^{2}-2\nabla f\nabla G\\
\nonumber
&&+2a\gamma(\alpha-1)|\nabla f|^{2}+a\gamma\alpha\Delta f-\gamma' F\\
\nonumber
&=&\gamma|f_{ij}+\frac{\varphi}{n}g_{ij}|^{2}+\Big[(\frac{2\varphi}{n}-\alpha')\frac{1}{\alpha}-\frac{\gamma'}{\gamma}\big]G
-\gamma\alpha^{2}n^{2}K^{2}\\
&&-2\nabla f\nabla G+2a\gamma(\alpha-1)|\nabla f|^{2}+a\gamma\alpha\Delta f.
\end{eqnarray}

Now, let $\varphi(r)$ be a $C^2$ function on $[0,\infty)$ such that
\begin{equation*}
\varphi(r)=\left\{\aligned
1 \quad if ~r\in[0,1],\\
0 \quad if ~r\in [2,\infty),
\endaligned\right.
\end{equation*}
and
$$0\leq \varphi(r)\leq 1, \quad \varphi'(r)\leq 0,\quad \varphi''(r)\leq 0, \quad
\frac{|\varphi'(r)|}{\varphi(r)}\leq C,$$
where $C$ is an absolute constant.
Define by
$$\phi(x,t)=\varphi(d(x,x_{0},t))=\varphi\left(\frac{d(x,x_{0},t)}{R}\right)
=\varphi\left(\frac{\rho(x,t)}{R}\right),$$
where $\rho(x,t)=d(x,x_{0},t)$.
By using  maximum principle, the argument of Calabi [2] allows us to
suppose that the function $\phi(x,t)$ with support in $B_{2R,T}$, is $C^2$ at the maximum point.
By utilizing the Laplacian theorem, we deduce that
\begin{eqnarray}\label{3.11}
\frac{|\nabla \phi|^{2}}{\phi}\leq \frac{C}{R^2},\quad
-\Delta \phi\leq \frac{C}{R^2}(1+\sqrt{K}R),
\end{eqnarray}

For any $0\leq T_{1}\leq T$, let $H=\phi G$ and $(x_{1},t_{1})$ be the point in $B_{2R,T_1}$
at which $H$ attains its maximum value. We can suppose that
$H$ is positive, because otherwise the proof is trivial.
Then at the point $(x_{1}, t_{1})$, we infer
\begin{equation}\label{3.12}
\left.\begin{array}{rcl} 0=\nabla (\phi G)=G\nabla \phi +\phi\nabla
G,\vspace{2mm}
\\
\Delta(\phi G)\leq 0,\vspace{2mm}
\\
 \partial_t (\phi G)\geq 0.\end{array}\right\}
\end{equation}
By the evolution formula of the geodesic length under the Ricci flow [6], we calculate
\begin{align*}
\phi_{t}G=&-G\phi'\left(\frac{\rho}{R}\right)\frac{1}{R}\frac{d\rho}{dt}
=G\phi'\left(\frac{\rho}{R}\right)\int_{\gamma_{t_1}}\mathrm{Ric}(S,S)ds\\
\leq &G\phi'\left(\frac{\rho}{R}\right)\frac{1}{R}K\rho\leq G\phi'\left(\frac{\rho}{R}\right)K_{2}\leq G\sqrt{C}K,
\end{align*}
where $\gamma_{t_1}$ is the geodesic connecting $x$ and $x_0$ under the metric
$g(t_1)$, $S$ is the unite tangent vector to $\gamma_{t_1}$, and $ds$ is the
element of the arc length.

All the following computations are at the point $(x_{1}, t_{1})$.
Since
\begin{eqnarray}\label{3.13}
\nonumber
|f_{ij}+\frac{\varphi}{n}\delta_{ij}|^{2}&\geq&\frac{1}{n}\Big(\mathbf{tr}|f_{ij}+\frac{\varphi}{n}\delta_{ij}| \Big)^{2}\\
\nonumber
&=&\frac{1}{n}\Big(\Delta f +\varphi\Big)\\
\nonumber
&=&\frac{1}{n}\Big[-\frac{1}{\alpha}F-\frac{1}{\alpha}(\alpha-1)|\nabla f|^{2}\Big]^{2}\\
&=&\frac{1}{\alpha^{2}n}\Big[\frac{G}{\gamma}+(\alpha-1)|\nabla f|^{2}\Big]^{2}.
\end{eqnarray}
and
\begin{eqnarray}\label{3.14}
\nonumber
\Delta f&=&f_{t}-|\nabla f|^{2}-af\\
&=&-\frac{F}{\alpha}-\frac{\alpha-1}{\alpha}|\nabla f|^{2}-\varphi<0.
\end{eqnarray}
\textbf{Case 1}~~$a\leq 0$.\quad Combining ~\eqref{3.14} with ~\eqref{3.10}, we have
\begin{eqnarray*}
\nonumber
(\Delta-\partial_{t}) G
&\geq &\gamma|f_{ij}+\frac{\varphi}{n}\delta_{ij}|^{2}+\Big[(\frac{2\varphi}{n}-\alpha')\frac{1}{\alpha}-\frac{\gamma'}{\gamma}\big]G
-\gamma\alpha^{2}n^{2}K^{2}\\
&&-2\nabla f\nabla G+2a\gamma(\alpha-1)|\nabla f|^{2}.
\end{eqnarray*}
Using ~\eqref{3.12} and ~\eqref{3.13}, we infer
\begin{eqnarray}\label{3.15}
\nonumber
0&\geq&(\Delta-\partial_{t}) (\phi G)\\
\nonumber
&=&G\Big(\Delta\phi-2\frac{|\nabla\phi|^2}{\phi}\Big)+\phi(\Delta-\partial_{t})G-\gamma G\phi_{t}\\
\nonumber
&\geq&G\Big(\Delta\phi-2\frac{|\nabla\phi|^2}{\phi}\Big)+\frac{\phi\gamma}{\alpha^{2}n}\Big[\frac{G}{\gamma}+(\alpha-1)|\nabla f|^{2}\Big]^{2}\\
\nonumber
&&+\Big[(\frac{2\varphi}{2}-\alpha')\frac{1}{\alpha}-\frac{\gamma'}{\gamma}\Big]\phi G-\gamma\phi\alpha^{2}n^{2}K^{2}-2\phi\nabla f\nabla G\\
&&+2a\phi\gamma(\alpha-1)|\nabla f|^{2}-G\sqrt{C}K.
\end{eqnarray}
Multiply $\phi$ to $(3.15)$, we have
\begin{eqnarray}\label{3.16}
\nonumber
0&\geq&\phi G\Big[\Delta\phi-2\frac{|\nabla\phi|^2}{\phi}+(\frac{2\varphi}{n}-\alpha')\frac{1}{\alpha}-\frac{\gamma'}{\gamma}\Big]
+\frac{\phi^{2}\gamma}{\alpha^{2}n}\Big[\frac{G}{\gamma}+(\alpha-1)|\nabla f|^{2}\Big]^{2}\\
\nonumber
&&-\gamma\phi^{2}\alpha^{2}n^{2}K^{2}-2\phi^{2}\nabla f\nabla G+2a\phi^{2}\gamma(\alpha-1)|\nabla f|^{2}
-\phi G\sqrt{C}K\\
\nonumber
&\geq&\phi G\Big[\Delta\phi-2\frac{|\nabla\phi|^2}{\phi}+(\frac{2\varphi}{n}-\alpha')\frac{1}{\alpha}-\frac{\gamma'}{\gamma}\Big]
+\frac{\phi^{2}G^{2}}{\alpha^{2}n\gamma}+\frac{\phi^{2}(\alpha-1)^{2}\gamma}{n\alpha^{2}}|\nabla f|^{4}\\
\nonumber
&&
+\frac{2\phi^{2}(\alpha-1)}{n\alpha^{2}}G|\nabla f|^{2}-\gamma\phi^{2}\alpha^{2}n^{2}K^{2}+2\phi G\nabla\phi\nabla f\\
&&+2a\phi^{2}\gamma(\alpha-1)|\nabla f|^{2}-\phi G\sqrt{C}K.
\end{eqnarray}
We use the fact
\begin{equation*}
\frac{2\phi^{2}(\alpha-1)}{n\alpha^{2}}G|\nabla f|^{2}+2\phi G\nabla\phi\nabla f
\geq -\frac{n\alpha^{2}}{2(\alpha-1)}\frac{|\nabla\phi|^{2}}{\phi}\phi G,
\end{equation*}
and
\begin{equation*}
\frac{\phi^{2}(\alpha-1)^{2}\gamma}{n\alpha^{2}}|\nabla f|^{4}+2a\phi^{2}\gamma(\alpha-1)|\nabla f|^{2}
\geq -na^{2}\alpha^{2}\gamma\phi^{2},
\end{equation*}
to ~\eqref{3.16}, we deduce that
\begin{eqnarray*}
\nonumber
0&\geq&\phi G\Big[\Delta\phi-2\frac{|\nabla\phi|^2}{\phi}+(\frac{2\varphi}{n}-\alpha')\frac{1}{\alpha}-\frac{\gamma'}{\gamma}
-\frac{n\alpha^{2}}{2(\alpha-1)}\frac{|\nabla\phi|^{2}}{\phi}-\sqrt{C}K\Big]
\\
\nonumber
&&+\frac{\phi^{2}G^{2}}{\alpha^{2}n\gamma}-\gamma\phi^{2}\alpha^{2}n^{2}K^{2}-na^{2}\alpha^{2}\gamma\phi^{2}\\
\nonumber
&\geq&\left[-\frac{C}{R^2}(1+\sqrt{k}R)-\frac{2C}{R^2}
+(\frac{2\varphi}{n}-\alpha')\frac{1}{\alpha}-\frac{\gamma'}{\gamma}
-\frac{n\alpha^{2}}{2(\alpha-1)}\frac{C}{R^2}-\sqrt{C}K\right]\phi G\\
&&+\frac{\phi^{2}G^{2}}{\alpha^{2}n\gamma}-\gamma\phi^{2}\alpha^{2}n^{2}K^{2}-na^{2}\alpha^{2}\gamma\phi^{2}.
\end{eqnarray*}
For the inequality $Ax^{2}-2Bx\leq C$, one has $x\leq \frac{2B}{A}+\left(\frac{C}{A}\right)^{\frac{1}{2}}$, where $A, B, C>0$.
Hence, we infer
\begin{eqnarray*}
\nonumber
\phi G(x,T_{1})&\leq& (\phi G)(x_{1},t_1)\\
\nonumber
&\leq& \Big\{n\gamma\alpha^{2}\left[\frac{C}{R^2}(1+\sqrt{K}R)+\frac{n\alpha^{2}}{2(\alpha-1)}\frac{C}{R^2}+\sqrt{C}K\right]\\
&&+n\gamma\alpha^{2}\left[\frac{\gamma'}{\gamma}-(\frac{2\varphi}{n}-\frac{\alpha'}{\alpha})\frac{1}{\alpha}\right]\\
&&+ n^{\frac{3}{2}}\gamma\alpha^{2}\phi K+na\alpha^{2}\gamma\phi\Big\}(x_{1},t_1).
\end{eqnarray*}
If $\gamma$ is nondecreasing which satisfies the system
\begin{equation*}
\left\{\aligned
\frac{\gamma'}{\gamma}-(\frac{2\varphi}{n}-\alpha')\frac{1}{\alpha}\leq 0,\\
\frac{\gamma\alpha^4}{\alpha-1}\leq C.
\endaligned\right.
\end{equation*}
Recall that $\alpha(t)$ and $\gamma(t)$ are non-decreasing and $t_{1}<T_{1}$. Hence, we have
\begin{eqnarray*}
\nonumber
\phi G(x,T_{1})&\leq& (\phi G)(x_{1},t_1)\\
\nonumber
&\leq& n\gamma(T_{1})\alpha^{2}(T_{1})\left[\frac{C}{R^2}\Big(1+\sqrt{K}R\Big)+K\right]+\frac{n^{2}C}{R^2}\\
&&+ n^{\frac{3}{2}}\gamma(T_{1})\alpha^{2}(T_{1}) K+na\alpha^{2}(T_{1})\gamma(T_{1}).
\end{eqnarray*}
Hence, we have for  $\phi\equiv 1$ on $B_{R,T}$,
\begin{eqnarray*}
F(x,T_{1})&\leq&n\alpha^{2}(T_{1})\left[\frac{C}{R^2}\Big(1+\sqrt{K}R\Big)+CK\right]+\frac{n^{2}C}{R^{2}\gamma(T_{1})}\\
&&+ n^{\frac{3}{2}}\alpha^{2}(T_{1}) K+na\alpha^{2}(T_{1}).
\end{eqnarray*}

If $\gamma$ is nondecreasing which satisfies the system
\begin{equation*}
\left\{\aligned
\frac{\gamma'}{\gamma}-(\frac{2\varphi}{n}-\alpha')\frac{1}{\alpha}\leq 0,\\
\frac{\gamma}{\alpha-1}\leq C.
\endaligned\right.
\end{equation*}
Recall that $\alpha(t)$ and $\gamma(t)$ are non-decreasing and $t_{1}<T_{1}$. Hence, we have
\begin{eqnarray*}
\nonumber
\phi G(x,T_{1})&\leq& (\phi G)(x_{1},t_1)\\
\nonumber
&\leq& n\gamma(T_{1})\alpha^{2}(T_{1})\left[\frac{C}{R^2}\Big(1+\sqrt{K}R\Big)+\frac{Cn\alpha^{4}}{R^2}+CK\right]\\
&&+ n^{\frac{3}{2}}\gamma(T_{1})\alpha^{2}(T_{1}) K+na\alpha^{2}(T_{1})\gamma(T_{1}).
\end{eqnarray*}
Hence, we have for  $\phi\equiv 1$ on $B_{R,T}$,
\begin{eqnarray*}
F(x,T_{1})&\leq&n\alpha^{2}(T_{1})\left[\frac{C}{R^2}\Big(1+\sqrt{K}R\Big)+K\right]+\frac{n^{2}C\alpha^{4}}{R^{2}\gamma(T_{1})}\\
&&+ n^{\frac{3}{2}}\alpha^{2}(T_{1}) K+na\alpha^{2}(T_{1}).
\end{eqnarray*}
Because $T_{1}$ is arbitrary in $0<T_{1}<T$,  the conclusion is valid.

\textbf{Case 2}~~$a\geq 0$.\quad It is not difficult to find $\Delta f\leq -\frac{F}{\alpha}$ form ~\eqref{3.14}.
Then, we have from ~\eqref{3.10}
\begin{eqnarray*}
\nonumber
(\Delta-\partial_{t}) G
&\geq &\gamma|f_{ij}+\frac{\varphi}{n}\delta_{ij}|^{2}+\Big[(\frac{2\varphi}{n}-\alpha')\frac{1}{\alpha}-\frac{\gamma'}{\gamma}\big]G
\\
&&-\gamma\alpha^{2}n^{2}K^{2}2\nabla f\nabla G-aG.
\end{eqnarray*}
Using ~\eqref{3.13} and ~\eqref{3.13}, we infer
\begin{eqnarray*}
\nonumber
0&\geq&(\Delta-\partial_{t}) (\phi G)\\
\nonumber
&=&G\Big(\Delta\phi-2\frac{|\nabla\phi|^2}{\phi}\Big)+\phi(\Delta-\partial_{t})G-\gamma G\phi_{t}\\
\nonumber
&\geq&G\Big(\Delta\phi-2\frac{|\nabla\phi|^2}{\phi}\Big)+\frac{\phi\gamma}{\alpha^{2}n}\Big[\frac{G}{\gamma}+(\alpha-1)|\nabla f|^{2}\Big]^{2}\\
\nonumber
&&+\Big[(\frac{2\varphi}{2}-\alpha')\frac{1}{\alpha}-\frac{\gamma'}{\gamma}\Big]\phi G-\gamma\phi\alpha^{2}n^{2}K^{2}-2\phi\nabla f\nabla G\\
&&-a\phi G-G\sqrt{C}K.
\end{eqnarray*}
Multiply $\phi$, we have
\begin{eqnarray}\label{3.17}
\nonumber
0&\geq&\phi G\Big[\Delta\phi-2\frac{|\nabla\phi|^2}{\phi}+(\frac{2\varphi}{n}-\alpha')\frac{1}{\alpha}-\frac{\gamma'}{\gamma}\Big]
+\frac{\phi^{2}\gamma}{\alpha^{2}n}\Big[\frac{G}{\gamma}+(\alpha-1)|\nabla f|^{2}\Big]^{2}\\
\nonumber
&&-\gamma\phi^{2}\alpha^{2}n^{2}K^{2}-2\phi^{2}\nabla f\nabla G-a\phi^{2}G
-\phi G\sqrt{C}K\\
\nonumber
&\geq&\phi G\Big[\Delta\phi-2\frac{|\nabla\phi|^2}{\phi}+(\frac{2\varphi}{n}-\alpha')\frac{1}{\alpha}-\frac{\gamma'}{\gamma}\Big]
+\frac{\phi^{2}G^{2}}{\alpha^{2}n\gamma}+\frac{2\phi^{2}(\alpha-1)}{n\alpha^{2}}G|\nabla f|^{2}\\
&&-\gamma\phi^{2}\alpha^{2}n^{2}K^{2}+2\phi G\nabla\phi\nabla f-a\phi^{2}G-\phi G\sqrt{C}K,
\end{eqnarray}
where we drop the term $\frac{\phi^{2}(\alpha-1)^{2}\gamma}{n\alpha^{2}}|\nabla f|^{4}$.
We use the fact
\begin{equation*}
\frac{2\phi^{2}(\alpha-1)}{n\alpha^{2}}G|\nabla f|^{2}+2\phi G\nabla\phi\nabla f
\geq -\frac{n\alpha^{2}}{2(\alpha-1)}\frac{|\nabla\phi|^{2}}{\phi}\phi G,
\end{equation*}
to ~\eqref{3.17}, we deduce that
\begin{eqnarray*}
\nonumber
0&\geq&\phi G\Big[\Delta\phi-2\frac{|\nabla\phi|^2}{\phi}+(\frac{2\varphi}{n}-\alpha')\frac{1}{\alpha}-\frac{\gamma'}{\gamma}
-\frac{n\alpha^{2}}{2(\alpha-1)}\frac{|\nabla\phi|^{2}}{\phi}-a\phi-\sqrt{C}K\Big]
\\
\nonumber
&&+\frac{\phi^{2}G^{2}}{\alpha^{2}n\gamma}-\gamma\phi^{2}\alpha^{2}n^{2}K^{2}\\
\nonumber
&\geq&\left[-\frac{C}{R^2}(1+\sqrt{k}R)-\frac{2C}{R^2}
+(\frac{2\varphi}{n}-\alpha')\frac{1}{\alpha}-\frac{\gamma'}{\gamma}
-\frac{n\alpha^{2}}{2(\alpha-1)}\frac{C}{R^2}-\sqrt{C}K\right]\phi G\\
&&+\frac{\phi^{2}G^{2}}{\alpha^{2}n\gamma}-\gamma\phi^{2}\alpha^{2}n^{2}K^{2}.
\end{eqnarray*}
For the inequality $Ax^{2}-2Bx\leq C$, one has $x\leq \frac{2B}{A}+\left(\frac{C}{A}\right)^{\frac{1}{2}}$, where $A, B, C>0$.
\begin{eqnarray*}
\nonumber
\phi G(x,T_{1})&\leq& (\phi G)(x_{1},t_1)\\
\nonumber
&\leq& \Big\{n\gamma\alpha^{2}\left[\frac{C}{R^2}(1+\sqrt{K}R)+\frac{n\alpha^{2}}{2(\alpha-1)}\frac{C}{R^2}+a\phi+\sqrt{C}K\right]\\
&&+n\gamma\alpha^{2}\left[\frac{\gamma'}{\gamma}-(\frac{2\varphi}{n}-\frac{\alpha'}{\alpha})\frac{1}{\alpha}\right]
+ n^{\frac{3}{2}}\gamma\alpha^{2}\phi K\Big\}(x_{1},t_1).
\end{eqnarray*}
If $\gamma$ is nondecreasing which satisfies the system
\begin{equation*}
\left\{\aligned
\frac{\gamma'}{\gamma}-(\frac{2\varphi}{n}-\alpha')\frac{1}{\alpha}\leq 0,\\
\frac{\gamma\alpha^4}{\alpha-1}\leq C.
\endaligned\right.
\end{equation*}
Recall that $\alpha(t)$ and $\gamma(t)$ are non-decreasing and $t_{1}<T_{1}$. Hence, we have
\begin{eqnarray*}
\nonumber
\phi G(x,T_{1})&\leq& (\phi G)(x_{1},t_1)\\
\nonumber
&\leq& n\gamma(T_{1})\alpha^{2}(T_{1})\left[\frac{C}{R^2}\Big(1+\sqrt{K}R\Big)+a\phi+K\right]+\frac{n^{2}C}{R^2}\\
&&+ n^{\frac{3}{2}}\gamma(T_{1})\alpha^{2}(T_{1}) K.
\end{eqnarray*}
Hence, we have for  $\phi\equiv 1$ on $B_{R,T}$,
\begin{eqnarray*}
\sup_{B_R}F(x,T_{1})&\leq&n\alpha^{2}(T_{1})\left[\frac{C}{R^2}\Big(1+\sqrt{K}R\Big)+a+CK\right]+\frac{n^{2}C}{R^{2}\gamma(T_{1})}\\
&&+ n^{\frac{3}{2}}\alpha^{2}(T_{1}) K.
\end{eqnarray*}

If $\gamma$ is nondecreasing which satisfies the system
\begin{equation*}
\left\{\aligned
\frac{\gamma'}{\gamma}-(\frac{2\varphi}{n}-\alpha')\frac{1}{\alpha}\leq 0,\\
\frac{\gamma}{\alpha-1}\leq C.
\endaligned\right.
\end{equation*}
Recall that $\alpha(t)$ and $\gamma(t)$ are non-decreasing and $t_{1}<T_{1}$. Hence, we have
\begin{eqnarray*}
\nonumber
\phi G(x,T_{1})&\leq& (\phi G)(x_{1},t_1)\\
\nonumber
&\leq& n\gamma(T_{1})\alpha^{2}(T_{1})\left[\frac{C}{R^2}\Big(1+\sqrt{K}R\Big)+\frac{C n\alpha^{4}}{R^2}+a\phi+CK\right]\\
&&+ n^{\frac{3}{2}}\gamma(T_{1})\alpha^{2}(T_{1}) K.
\end{eqnarray*}
Hence, we have for  $\phi\equiv 1$ on $B_{R,T}$,
\begin{eqnarray*}
\nonumber
F(x,T_{1})&\leq&n\alpha^{2}(T_{1})\left[\frac{C}{R^2}\Big(1+\sqrt{K}R\Big)+a+K\right]+\frac{n^{2}C\alpha^4}{R^{2}\gamma(T_{1})}\\
&&+ n^{\frac{3}{2}}\alpha^{2}(T_{1}) K.
\end{eqnarray*}
Because $T_1$ is arbitrary in $0<T_{1}<T$,  the conclusion is valid. This proof is complete.

\end{proof}

\section{\textbf{Harnack Inequalities}}
In this section, as application of main theorems, some Harnack inequalities are derived.

\begin{thm}
Let $(M^{n}, g(x,t))_{t\in [0,T]}$ be a complete solution to the Ricci flow \eqref{1.7}. Suppose that
$|\mathrm{Ric}|\leq K$ for some $K>0$, and all $(x,t)\in M^{n}\times[0,T]$. Assume that
$u(x,t)$ is a positive solution for ~\eqref{1.6}.
Let $h(x,t)$ be a function defined
on $M^{n}\times[0,T]$ which is $C^1$ in $t$ and $C^2$ in $x$, satisfying $|\nabla h|^{2}\leq \delta_{2}h$ and
$\Delta h\geq -\delta_{3}$ on $M^{n}\times[0,T]$ for some positive constants $\delta_{2}$ and $\delta_{3}$.
Then for all $(x_{1},t_{1})\in M^{n}\times(0,T)$ and $(x_{2},t_{2})\in M^{n}\times(0,T)$
such that $t_{1}<t_{2}$, we have
\begin{align*}
u(x_{2},t_{2})&\leq \left\{\aligned
u(x_{1},t_{1})
\times\exp\big(\Gamma(t_{1},t_{2},\delta_{1},\delta_{2},\delta_{3},\overline{u}_{1})\big),\quad l\leq 1,\\
u(x_{1},t_{1})
\times\exp\big(\Lambda(t_{1},t_{2},\delta_{1},\delta_{2},\delta_{3},\overline{u}_{1})\big),\quad l> 1,
\endaligned\right.
\end{align*}
where
\begin{align*}
&\Gamma(t_{1},t_{2},\delta_{1},\delta_{2},\delta_{3},\overline{u}_{1})\\
&=
\int_{0}^{1}\frac{|\gamma'(s)|^4}{2(t_{2}-t_{1})^2}ds+\int_{t_{1}}^{t_{2}}\frac{\alpha^{2}(t)}{32}dt
+\int_{t_{1}}^{t_{2}}[\varphi+C\alpha^{2}(K+\mu_{1})+\delta_{1}\overline{u}_{1}]dt,\\
&\Lambda(t_{1},t_{2},\delta_{1},\delta_{2},\delta_{3},\overline{u}_{1})\\
&=
\int_{0}^{1}\frac{|\gamma'(s)|^4}{2(t_{2}-t_{1})^2}ds+\int_{t_{1}}^{t_{2}}\frac{\alpha^{2}(t)}{32}dt
+\int_{t_{1}}^{t_{2}}[\varphi+C\alpha^{2}(K+\mu_{2})]dt\\
&\qquad+\int_{t_{1}}^{t_{2}}\Big[\delta\sqrt{\frac{l\alpha-1}{l-1}}\sqrt{\overline{u}_{2}\delta_{2}}
+\alpha^{\frac{3}{2}}\sqrt{n(l-1)\varphi\delta_{1}}\Big]dt.
\end{align*}
\end{thm}
\begin{proof}
Firstly, the estimate in Corollary $2.2$ can be written as
\begin{align}\label{4.1}
\nonumber
&\frac{|\nabla u(x,t)|^{2}}{u^{2}(x,t)}-\alpha(t)\frac{u_{t}(x,t)}{u(x,t)}+\alpha(t)h(x,t)u^{l-1}(x,t)\\
&\leq
\left\{\aligned
\alpha\varphi+C\alpha^{2}(K+\mu_{1}),\quad\quad\quad\quad\quad\quad l\leq 1,\\
\alpha\varphi+C\alpha^{2}(K+\mu_{2})
+\delta\sqrt{\frac{l\alpha-1}{l-1}}\sqrt{\overline{u}_{2}\delta_{2}}\\
\quad\quad+\alpha^{\frac{3}{2}}\sqrt{n(l-1)\varphi\delta_{1}},\quad l>1,
\endaligned\right.
\end{align}
where $\mu_{1}=\sqrt{\overline{u}_{1}\delta_{3}}+\overline{u}_{1}\delta_{1}+\sqrt{(2-l)\overline{u}_{1}\delta_{2}}$ and
$\mu_{2}=(l-1)\overline{u}_{2}\delta_{1}+\sqrt{\overline{u}_{2}\delta_{3}}$.

Now we only prove the conclusion for $l\leq 1$.\\
Define $l(s)=\log(\gamma(s), (1-s)t_{2}+st_{1})$. Obviously, we infer that $l(0)=\log u(y, t_2)$
and $l(1)=\log u(x, t_1)$. Direct calculation shows
\begin{eqnarray*}
\frac{\partial l(s)}{\partial s}&=&(t_{2}-t_{1})\left(\frac{\nabla u}{u}\frac{\gamma'(s)}{t_{2}-t_{1}}
-\frac{u_t}{u}\right)\\
&\leq &(t_{2}-t_{1})\left[\frac{\nabla u}{u}\frac{\gamma'(s)}{t_{2}-t_{1}}-\frac{1}{\alpha(t)}\frac{|\nabla u|^{2}}{u^{2}}
-h(x,t)u^{l-1}+\varphi+C\alpha(K+\mu_{1})\right]\\
&\leq &\frac{\alpha(t)}{4}\frac{|\gamma'(s)|^2}{t_{2}-t_{1}}+(t_{2}-t_{1})[\varphi+C\alpha(K+\mu)+\delta_{1}\overline{u}_{1}].
\end{eqnarray*}
Integrating above inequality over $\gamma(s)$, we obtain
\begin{eqnarray*}
\log\frac{u(x,t_1)}{u(y,t_2)}&=&\int_{0}^{1}\frac{\partial l(s)}{\partial s}ds\\
&\leq &\int_{0}^{1}\left[\frac{\alpha(t)}{4}\frac{|\gamma'(s)|^2}{t_{2}-t_{1}}
+(t_{2}-t_{1})[\varphi+C\alpha(K+\mu_{1})+\delta_{1}\overline{u}_{1}]\right]ds\\
&\leq &\int_{0}^{1}\frac{|\gamma'(s)|^4}{2(t_{2}-t_{1})^2}ds+\int_{t_{1}}^{t_{2}}\frac{\alpha^{2}(t)}{32}dt
\\
&&
+\int_{t_{1}}^{t_{2}}[\varphi+C\alpha(K+\mu)+\delta_{1}\overline{u}_{1}]dt.
\end{eqnarray*}
The proof is complete.
\end{proof}

We also derive an Harnack inequality for the equation \eqref{1.6}. The proof is similar to Theorem $4.1$, so we omit it.
\begin{thm}
Let $(M^{n}, g(x,t))_{t\in [0,T]}$ be a complete solution to the Ricci flow \eqref{1.7}. Suppose that
$|\mathrm{Ric}|\leq K$ for some $K>0$, and all $(x,t)\in M^{n}\times[0,T]$. Assume that
$u(x,t)$ is a positive solution for ~\eqref{1.6}.
Then for all $(x_{1},t_{1})\in M^{n}\times(0,T)$ and $(x_{2},t_{2})\in M^{n}\times(0,T)$
such that $t_{1}<t_{2}$, we have
\begin{align*}
u(x_{2},t_{2})&\leq u(x_{1},t_{1})\\
&\times\exp\left(\int_{0}^{1}\frac{|\gamma'(s)|^4}{2(t_{2}-t_{1})^2}ds+\int_{t_{1}}^{t_{2}}\frac{\alpha^{2}(t)}{32}dt
+\int_{t_{1}}^{t_{2}}[\varphi+C\alpha(K+\mu)+\delta_{1}\overline{u}_{1}]dt\right)
\end{align*}
\end{thm}

\section{\textbf{Application to heat equation}}
\

According to Theorem $2.1$ and Theorem $3.1$, we derive corresponding gradient
estimates and Harnack inequalities to the heat equation under Ricci flow
\begin{thm}
Let $(M^{n}, g(t))_{t\in [0,T]}$ be a complete solution to the Ricci flow ~\eqref{1.7}.
  Assume that $|\mathrm{Ric}(x,t)|\leq K$ for some $K>0$ and all $t\in [0,T]$.
  Suppose that there exist three functions $\alpha(t)$, $\varphi(t)$ and $\gamma(t)$
  satisfy the following conditions (C1), (C2), (C3) and (C4).

Given $x_{0}\in M$ and $R>0$, let $u(x,t)$ be a positive solution of the heat equation
\begin{equation}\label{5.1}
u_{t}=\Delta u,
\end{equation}
in the cube $B_{2R,T}:=\{(x,t)|d(x,x_{0},t)\leq 2R, 0\leq t\leq T\}$, where $c$ is a constant.

If $\frac{\gamma\alpha^4}{\alpha-1}\leq C_{1}$ for some constant $C_1$,
then
\begin{eqnarray*}
\frac{|\nabla u|^2}{u^2}-\alpha\frac{u_{t}}{u}
\leq
C\alpha^{2}\left(\frac{1}{R^2}+\frac{\sqrt{K}}{R}+K\right)+\frac{Cn^2}{R^{2}\gamma}+\alpha\varphi.
\end{eqnarray*}

If $\frac{\gamma}{\alpha-1}\leq C_{2}$ for some constant $C_2$,
then
\begin{eqnarray*}
\frac{|\nabla u|^2}{u^2}-\alpha\frac{u_{t}}{u}
\leq
C\alpha^{2}\left(\frac{1}{R^2}+\frac{\sqrt{K}}{R}+K\right)+\frac{Cn^{2}\alpha^4}{R^{2}\gamma}+\alpha\varphi.
\end{eqnarray*}
where $C$ is a constant.
\end{thm}

\begin{cor}  Let $(M^{n}, g(t))_{t\in [0,T]}$ be a complete solution to the Ricci flow ~\eqref{1.7}.
  Assume that $|\mathrm{Ric}(x,t)|\leq K$ for some $K>0$ and all $t\in [0,T]$.
Given $x_{0}\in M$ and $R>0$, let $u(x,t)$ be a positive solution of the heat equation  ~\eqref{5.1}
in the cube $B_{2R,T}:=\{(x,t)|d(x,x_{0},t)\leq 2R, 0\leq t\leq T$\}. Then the following special estimates are valid.

1. Li-Yau type:
$$\alpha(t)=constant,\quad \varphi(t)=\frac{n}{t}+\frac{nK\alpha^2}{\alpha-1},\gamma(t)=t^{\theta}
\quad  with \quad 0<\theta\leq 2.$$
\begin{eqnarray*}
\nonumber
\frac{|\nabla u|^2}{u^2}-\alpha\frac{u_{t}}{u}
&\leq&
C\alpha^{2}\left[\frac{1}{R^2}(1+\sqrt{K}R)+\frac{\alpha^2}{\alpha-1}\frac{1}{R^2}+K\right]\\
&&+\alpha\varphi
+n^{\frac{3}{2}}\alpha^{2}K.
\end{eqnarray*}

2. Hamilton type:
$$\alpha(t)=e^{2Kt},\quad \varphi(t)=\frac{n}{t}e^{4Kt},\quad \gamma(t)=te^{2Kt}.$$
\begin{eqnarray*}
\nonumber
\frac{|\nabla u|^2}{u^2}-\alpha\frac{u_{t}}{u}
&\leq&
C\alpha^{2}\left[\frac{1}{R^2}(1+\sqrt{K}R)+K\right]+\frac{C\alpha^4}{R^{2}te^{2Kt}}\\
&&+\alpha\varphi
+n^{\frac{3}{2}}\alpha^{2}K.
\end{eqnarray*}

3. Li-Xu type:
\begin{align*}
&\alpha(t)=1+\frac{\sinh(Kt)\cosh(Kt)-Kt}{\sinh^{2}(Kt)},\quad \varphi(t)=2nK[1+\coth(Kt)],\\
&\gamma(t)=\tanh(Kt).
\end{align*}
\begin{eqnarray*}
\nonumber
\frac{|\nabla u|^2}{u^2}-\alpha\frac{u_{t}}{u}
&\leq&
C\alpha^{2}\left[\frac{1}{R^2}(1+\sqrt{K}R)+K\right]+\frac{C}{R^{2}\tanh(Kt)}\\
&&+\alpha\varphi
+n^{\frac{3}{2}}\alpha^{2}K.
\end{eqnarray*}

4. Linear Li-Xu type:
$$\alpha(t)=1+2Kt,\quad \varphi(t)=\frac{n}{t}+nK(1+2Kt+\mu Kt), \gamma(t)=Kt \quad with \quad \mu\geq \frac{1}{4}.$$
\begin{eqnarray*}
\nonumber
\frac{|\nabla u|^2}{u^2}-\alpha\frac{u_{t}}{u}
&\leq&
C\alpha^{2}\left[\frac{1}{R^2}(1+\sqrt{K}R)+K\right]+\frac{C\alpha^4}{R^{2}Kt}\\
&&+\alpha\varphi
+n^{\frac{3}{2}}\alpha^{2}K.
\end{eqnarray*}
\end{cor}

Let $R\rightarrow\infty$, a global estimate is derived.
\begin{cor}
Let $(M^{n}, g(t))_{t\in [0,T]}$ be a complete solution to the Ricci flow ~\eqref{1.7}.
  Assume that $|\mathrm{Ric}(x,t)|\leq K$ for some $K>0$ and all $t\in [0,T]$.
  Suppose that there exist three functions $\alpha(t)$, $\varphi(t)$ and $\gamma(t)$
  satisfy the following conditions (C1), (C2), (C3) and (C4).

Given $x_{0}\in M$ and $R>0$, let $u(x,t)$ be a positive solution of the heat equation
$(5.2)$
in the cube $M^{n}\times [0,T]$.
Then
\begin{eqnarray*}
\frac{|\nabla u|^2}{u^2}-\alpha\frac{u_{t}}{u}
\leq
C\alpha^{2}K+\alpha\varphi,
\end{eqnarray*}
where $C$ is a constant.
\end{cor}

Using theorem $4.1$, we derive a Harnack inequality.

\begin{cor}
(Harnack Inequality)\quad
Let $(M^{n}, g(t))_{t\in [0,T]}$ be a complete solution to the Ricci flow \eqref{1.7}. Suppose that
$|\mathrm{Ric}|\leq K$ for some $K>0$, and all $(x,t)\in M^{n}\times[0,T]$. Assume that
$u(x,t)$ is a positive solution for ~\eqref{5.1}.
Then for all $(x_{1},t_{1})\in M^{n}\times(0,T)$ and $(x_{2},t_{2})\in M^{n}\times(0,T)$
such that $t_{1}<t_{2}$, we have
\begin{align*}
u(x_{2},t_{2})&\leq u(x_{1},t_{1})\\
&\times\exp\left(\int_{0}^{1}\frac{|\gamma'(s)|^4}{2(t_{2}-t_{1})^2}ds+\int_{t_{1}}^{t_{2}}\frac{\alpha^{2}(t)}{32}dt
+\int_{t_{1}}^{t_{2}}[\varphi+C\alpha K]dt\right)
\end{align*}
\end{cor}

\section{\textbf{Appendix}}
We will check some special functions $\alpha(t)>1$, $\varphi(t)>0$ and $\gamma(t)>0$   satisfy the following two systems
\begin{equation}\label{6.1}
\left\{\aligned
\frac{2\varphi}{n}-2\alpha K\geq (\frac{2\varphi}{n}-\alpha')\frac{1}{\alpha},\\
\frac{2\varphi}{n}-\alpha'>0,\\
\frac{\varphi^2}{n}+\alpha\varphi'\geq 0.
\endaligned\right.
\end{equation}
and
\begin{equation}\label{6.2}
\left\{\aligned
\frac{\gamma'}{\gamma}-(\frac{2\varphi}{n}-\alpha')\frac{1}{\alpha}\leq 0,\\
\frac{\gamma\alpha^4}{\alpha-1}\leq C, ~or~ \frac{\gamma}{\alpha-1}\leq C.
\endaligned\right.
\end{equation}
Besides, $\alpha(t)$ and $\gamma(t)$ are non-decreasing.

$(1)$ Let $\alpha(t)=1+2Kt$,  $\varphi(t)=\frac{n}{t}+nK(1+2Kt+\mu Kt)$ $(\mu\geq \frac{1}{4})$ and $\gamma(t)=Kt$.

One can has
\begin{align*}
(\mathrm{i})\quad &\frac{2\varphi}{n}-\alpha'=\frac{2}{t}+2K(1+2Kt+\mu Kt)-2K>0,\\
(\mathrm{ii})\quad &\frac{\varphi^2}{n}+\alpha\varphi'
=\frac{n}{t^2}+nK^{2}(1+2Kt+\mu Kt)^{2}+\frac{2nK}{t}(1+2Kt+\mu Kt)\\
&+(1+2Kt)(-\frac{n}{t^2}+2nK^{2}+n\mu K^2)\\
&=nK^{2}(1+2Kt+\mu Kt)^{2}+\frac{2nK}{t}(2Kt+\mu Kt)\\
&+(1+2Kt)(2nK^{2}+n\mu K^2)>0,\\
(\mathrm{iii})\quad &\frac{2\varphi}{n}-2\alpha K-(\frac{2\varphi}{n}-\alpha')\frac{1}{\alpha}\\
&=\frac{2}{t}+2K(1+2Kt+\mu Kt)-2K(1+2Kt)\\
&-\Big[\frac{2}{t}+2K(1+2Kt+\mu Kt)-2K\Big]\cdot\frac{1}{1+2Kt}\\
&=\frac{4Kt(\mu K^{2}t^{2}-Kt+1)}{t(1+2Kt)}\geq 0,\quad for \quad \mu\geq \frac{1}{4}.
\end{align*}
Hence, $\alpha(t)=1+2Kt$,  $\varphi(t)=\frac{n}{t}+nK(1+2Kt+\mu Kt)$ $(0<\mu\leq \frac{1}{4})$ satisfy  system (6.1).

On the other hand, one has
\begin{eqnarray*}
&&\frac{\gamma'}{\gamma}-(\frac{2\varphi}{n}-\alpha')\frac{1}{\alpha}\\
&=&\frac{1}{t}-\left(\frac{2}{t}+2K(1+2Kt+\mu Kt)-2K\right)\frac{1}{1+2Kt}\\
&=&\frac{1}{t(1+2Kt)}\left[-(4K^{2}+2K\mu)t^{2}+2Kt-1\right]\\
&=&\frac{1}{t(1+2Kt)}\left[-(3K^{2}+2K\mu)t^{2}-(Kt-1)^{2}\right]\\
&\leq & 0,\quad for \quad t\geq 0.
\end{eqnarray*}
and $\frac{\gamma}{\alpha-1}=\frac{1}{2}$. So, (6.2) is also satisfied.

$(2)$ $\alpha(t)=e^{2Kt}$,  $\varphi(t)=\frac{n}{t}e^{4Kt}$ and $\gamma(t)=te^{2Kt}$, where $(0<Kt\leq 1)$ .
Direct calculation gives
\begin{align*}
(\mathrm{i})\quad &\frac{2\varphi}{n}-\alpha'=\frac{2}{t}e^{2Kt}(e^{2Kt}-Kt)>0,\\
(\mathrm{ii})\quad &\frac{\varphi^2}{n}+\alpha\varphi'=\frac{n}{t^2}e^{6Kt}(e^{2Kt}-1+4Kt)>0,\\
(\mathrm{iii})\quad &\frac{2\varphi}{n}-2\alpha K-(\frac{2\varphi}{n}-\alpha')\frac{1}{\alpha}=\frac{2}{t}e^{4Kt}-2Ke^{2Kt}-\frac{2}{t}e^{2Kt}+2K\\
&=(e^{2Kt}-1)(\frac{2}{t}e^{2Kt}-2K)\geq 0.
\end{align*}
Hence, $\alpha(t)=e^{2Kt}$ and  $\varphi(t)=\frac{n}{t}e^{4Kt}$ satisfy  system (6.1).

Besides, we have
\begin{eqnarray*}
&&\frac{\gamma'}{\gamma}-(\frac{2\varphi}{n}-\alpha')\frac{1}{\alpha}\\
&=&\frac{1+2Kt}{t}-\left(\frac{2}{t}e^{2Kt}-2K\right)\\
&=&\frac{1}{t}(1+4Kt-2e^{2Kt})\\
&\leq & 0,\quad for \quad t\geq 0.
\end{eqnarray*}
and as $t\rightarrow 0^{+}$, $\frac{\gamma}{\alpha-1}=\frac{te^{2Kt}}{e^{2Kt}-1}\rightarrow\frac{1}{2K}$. This
implies $\frac{\gamma}{\alpha-1}\leq C$. So, (6.2) is also satisfied.

$(3)$ $\alpha(t)=1+\frac{\sinh(Kt)\cosh(Kt)-Kt}{\sinh^{2}(Kt)}$,  $\varphi(t)=2nK[1+\coth(Kt)]$ and $\gamma(t)=\tanh(Kt)$.
Direct calculation gives
\begin{align*}
(\mathrm{i})\quad &\frac{2\varphi}{n}-\alpha'=4K[1+\coth(Kt)]-2K+2K\coth^{2}(Kt)-\frac{2K^{2}t}{\sinh^{2}(Kt)}\coth(Kt)\\
&\quad\quad \quad\quad=2K+2K(1+\alpha)\coth(Kt)>0,\\
(\mathrm{ii})\quad&\alpha(\frac{2\varphi}{n}-2\alpha K)-(\frac{2\varphi}{n}-\alpha')\\
&=4K\alpha[1+\coth(Kt)]-2K\alpha^{2}-[2K+2K(1+\alpha)\coth(Kt)]\\
&=2K\alpha\Big[1+\coth(Kt)+\frac{Kt}{\sinh^{2}(Kt)}\Big]-[2K+2K(1+\alpha)\coth(Kt)]\\
&=2K(\alpha-1)\frac{Kt}{\sinh^{2}(Kt)}>0,\\
(\mathrm{iii})\quad &\frac{\varphi^2}{n}+\alpha\varphi'
=\frac{2nK^{2}}{\sinh^{2}(Kt)}\Big[2(1+\coth(Kt))^{2}\sinh^{2}(Kt)-\alpha\Big]\\
&=\frac{2nK^{2}}{\sinh^{2}(Kt)}\Big[2e^{2Kt}-1-\frac{e^{4Kt}-1-4Kte^{2Kt}}{(e^{2Kt}-1)^2}\Big]\\
&=\frac{4nK^{2}e^{2Kt}}{(e^{2Kt}-1)^2\sinh^{2}(Kt)}\Big[e^{4Kt}-3e^{2Kt}+2+4Kt\Big].
\end{align*}
Let $f(x)=e^{4x}-3e^{2x}+2+4x$ with $x\leq 0$. Obviously, $f(0)=0$ and
\begin{eqnarray*}
f'(x)&=&4e^{4x}-6e^{2x}+4>0.
\end{eqnarray*}
Then we get $f(x)>0$ for $x>0$. Hence, we have
\begin{eqnarray*}
&&(\frac{2\varphi}{n}-\alpha')\varphi+\alpha\varphi'+\alpha'\varphi-\frac{\varphi^2}{n}\\
&=&\frac{4nK^{2}e^{2Kt}}{(e^{2Kt}-1)^2\sinh^{2}(Kt)}\Big[e^{4Kt}-3e^{2Kt}+2+4Kt\Big]>0.
\end{eqnarray*}
Hence, $\alpha(t)=1+\frac{\sinh(Kt)\cosh(Kt)-Kt}{\sinh^{2}(Kt)}$ and $\varphi(t)=2nK[1+\coth(Kt)]$ satisfy  system (6.1).

On the other hand, as $t\rightarrow 0$, we have $\frac{\gamma\alpha^4}{\alpha-1}\rightarrow 2$;
 $\frac{\gamma\alpha^4}{\alpha-1}\rightarrow 1$ for $t\rightarrow \infty$. These imply $\frac{\gamma\alpha^4}{\alpha-1}\leq C$, here $C$ is a universal constant.\\
 Besides, we have
\begin{eqnarray*}
&&\frac{\gamma'}{\gamma}-(\frac{2\varphi}{n}-\alpha')\frac{1}{\alpha}\\
&=&\frac{1}{\alpha}\left[\frac{K\alpha}{\sinh(Kt)\cosh(Kt)}-2K-2K(1+\alpha)\coth(Kt)\right]\\
&=&\frac{1}{\alpha}\left[\frac{K}{\sinh(Kt)\cosh(Kt)}[\alpha-2(1+\alpha)\cosh^{2}(Kt)]-2K\right]\\
&=&\frac{1}{\alpha}\left[\frac{K}{\sinh(Kt)}[\alpha(1-2\cosh(Kt))-2\cosh(Kt)]-2K\right]\\
&\leq & 0,\quad for \quad t\geq 0.
\end{eqnarray*}
So, (6.2) is also satisfied.

$(4)$ $\alpha(t)=constant$,  $\varphi(t)=\frac{\alpha n}{t}+\frac{nK\alpha^2}{\alpha-1}$ and $\gamma(t)=t^{\theta}$ with $0<\theta\leq 2$.
Direct calculation gives
\begin{align*}
(\mathrm{i})\quad &\frac{2\varphi}{n}-\alpha'=\frac{2}{n}\Big[\frac{\alpha n}{t}+\frac{nK\alpha^2}{\alpha-1}\Big]>0,\\
(\mathrm{ii})\quad &\frac{\varphi^2}{n}+\alpha\varphi'=\frac{n\alpha^2}{t^2}+\frac{n^{2}K^{2}\alpha^4}{n(\alpha-1)^2}+\frac{2nK\alpha^{2}}{(\alpha-1)t}-\frac{n\alpha^{2}}{t^2}>0,\\
(\mathrm{iii})\quad &(\frac{2\varphi}{n}-2\alpha K)-(\frac{2\varphi}{n}-\alpha')\frac{1}{\alpha}\\
&=\frac{2\varphi}{n\alpha}(\alpha-1)-2K\alpha\\
&\geq\frac{2}{n\alpha}(\alpha-1)\frac{nK\alpha^2}{\alpha-1}-2K\alpha=0.
\end{align*}
Hence, $\alpha(t)=constant$, and $\varphi(t)=\frac{\alpha n}{t}+\frac{nK\alpha^2}{\alpha-1}$ satisfy  system (6.1).

 On the other hand, we have
\begin{eqnarray*}
\frac{\gamma'}{\gamma}-(\frac{2\varphi}{n}-\alpha')\frac{1}{\alpha}
&=&\frac{\theta}{t}-\frac{2}{t}-\frac{2K\alpha}{\alpha-1}\\
&\leq & 0,\quad for \quad t\geq 0 \quad and \quad 0<\theta\leq 2.
\end{eqnarray*}
So, (6.2) is also satisfied.

\section{\textbf{Acknowledgement}}
We are grateful to Professor Jiayu Li for his encouragement. We also thank Professor Qi S Zhang for introduction of this
problem in the summer course.

\end{document}